\documentclass[12pt]{article}
\usepackage{amsthm}
\usepackage{amsmath,bm}
\usepackage{enumerate}
\usepackage{amssymb}
\usepackage{latexsym}
\usepackage{amsfonts}
\usepackage{color}
\usepackage{secdot}
\usepackage{mathrsfs}
\usepackage{subfigure}
\usepackage{epsfig}
\usepackage{natbib}
\usepackage{rotating} 
\newtheorem{theorem}{Theorem}[section]

\newtheorem{counterexample}{Counterexample}[section]
\newtheorem{example}{Example}[section]

\newtheorem{definition}{Definition}[section]
\newtheorem{lemma}{Lemma}[section]
\newtheorem{remark}{Remark}[section]
\newtheorem{proposition}{Proposition}[section]
\newtheorem{assumption}{Assumption}[section]

\usepackage{hyperref}
\usepackage{multirow}
\usepackage{multicol}
\usepackage{booktabs,tabularx}
\newcolumntype{x}{>{\raggedright\arraybackslash}X}
\makeatletter \oddsidemargin  0.95in \evensidemargin 0.95in
\begin{document}

\title{\bf Stochastic comparisons of finite mixtures with general exponentiated location-scale distributed components}
\author{\small {Raju {\bf Bhakta}$^{1}$\thanks {E-mail address (corresponding author): bhakta.r93@gmail.com, Raju.Bhakta@niituniversity.in},~Kaushik {\bf Gupta}$^{2}$\thanks {E-mail address: kaushikgupta2001@gmail.com},~Ghobad Saadat Kia {\bf (Barmalzan)}$^{3}$\thanks {E-mail address: ghsaadatkia@gmail.com, Gh.saadatkia@kut.ac.ir}~and~Suchandan {\bf {Kayal}$^{4}$\thanks {E-mail address: kayals@nitrkl.ac.in,~suchandan.kayal@gmail.com}}}}
\date{}
\maketitle
\noindent{\it$^{1}$Department of Mathematics and Basic Sciences, NIIT University, NH-8, Delhi-Jaipur Highway, Neemrana, Rajasthan-301705, India.}\\
\noindent{\it$^{2,4}$Department of Mathematics, National Institute of Technology Rourkela, Sector 1, Rourkela, Sundargarh, Odisha-769008, India.}\\
\noindent{\it$^{3}$Department of Basic Science, Kermanshah University of Technology, Kermanshah Province, Kermanshah, Imam Khomeyni Highway, 83GW+3H3, IranKermanshah, Iran.}	
\vspace*{.5cm}
\begin{center}
\noindent{\bf Abstract}	
\end{center}
In this paper, we study stochastic ordering results between two finite mixtures with single and multiple outliers, assuming subpopulations follow general exponentiated location-scale distributions. For single-outlier mixtures, several sufficient conditions are derived under which the mixture variables are ordered in the usual stochastic, reversed hazard rate, and likelihood ratio orders, using majorization concepts. For multiple-outlier mixtures, results are obtained for the reversed hazard rate, likelihood ratio, and ageing faster orders in reversed hazard rate. Numerical examples and counterexamples are presented to illustrate and support the established theoretical findings.
\begin{table*}[h!]\label{table:1}
\centering
\begin{tabular}{llll}
\toprule
&~~~~~~~~~~~~~~~~~~~~~~~~~~~~~~\textbf{Abbreviations}&&\\
\midrule
RV & Random Variable & ELS & Exponentiated Location-Scale\\
MRV & Mixture Random Variable &	AFO & Ageing Faster Order\\
FM & Finite Mixture & PDF & Probability Density Function\\
FMRV & Finite Mixture Random Variable & CDF & Cumulative Distribution Function\\
MM & Mixture Model & SF & Survival Function\\
FMM & Finite Mixture Model & RHRF & Reversed Hazard Rate Function\\ 
\bottomrule
\end{tabular}
\end{table*}	
\section{Introduction}\label{section1}
For a considerable amount of time, academics have focused on the study of various problems in probability and statistics assuming that the population is homogeneous. But in reality, populations are not inherently homogeneous; rather, they are composed of information coming from several heterogeneous subpopulations. FMMs are widely employed in several scientific and technological fields for better modelling and inferences. \cite{finkelstein2008failure} and \cite{cha2013failure} discuss how various batches of components might be heterogeneous due to factors such as environment, materials, machinery, and human error. Note that the mass-product components form a heterogeneous population with various subpopulations. A randomly selected batch of components may belong to one of these subpopulations with varying probability. The FMM allows  modelling the lifetime data from a limited number of subpopulations.
	
\cite{finkelstein2006mixture} found that ordered mixing distributions deal with ordering mixture failure rates. \cite{navarro2008mean} studied ordering properties for coherent systems in which components are dependent and identically distributed. The results are independent of the distribution of the common components.
\cite{navarro2013stochastic} defined suitable criteria in order to compare the hazard rate and likelihood ratio orders of generalized mixtures. \cite{balakrishnan2014stochastic} investigates the stochastic comparison of series and parallel systems with s-independent heterogeneous generalized exponential components. \cite{navarro2016stochastic} define suitable criteria for comparing the likelihood ratio and hazard rate orders of generalized mixtures. Then \cite{amini2017stochastic} compared conventional stochastic orders to two classical FMMs. They conducted stochastic comparisons between two FMMs with ordered baseline RVs. The baseline distributions are organized in a stochastic order, such as $F_1 \geq_{hr}\ldots\geq_{hr} F_n$. \cite{navarro2017stochastic} discovered the necessary and sufficient criteria for stochastic comparisons of generalized distorted distributions with ordered baseline distribution functions. The findings for generalized distorted distributions may be applied to FMs with component-wise order requirements like $F_1 \geq_{st}\ldots\geq_{st} F_n$. \cite{hazra2018stochastic} employed multivariate chain majorization to compare stochastic results for two FMs using proportional reversed hazard rate, proportional hazard rate,  or accelerated lifetime  models for RVs. \cite{barmalzan2021comparisons} compare two finite $\alpha$ MMs using various stochastic orders. \cite{sattari2021stochastic} develop some comparison results when  generalized Lehmann a distribution is followed by components of two FMMs. \cite{panja2021stochastic} compared two FMMs whenever the components have proportional odds, proportional reversed hazard rate and proportional hazards rate models. \cite{barmalzan2022orderings} examined the outcomes of a stochastic comparison in relation to both the usual stochastic order and the reversed hazard rate order when components are from the location-scale family with two FMMs. \cite{bhakta2024stochasticcomparisons} studied stochastic comparison results with respect to reversed hazard rate order, hazard rate order and usual stochastic order  after taking two FMMs when components follows the general parametric family. Recently, \cite{bhakta2024stochasticorderings} established stochastic comparison with respect to usual stochastic order, reversed hazard rate order, likelihood ratio order and AFO in terms of reversed hazard rate order for components having two FMMs with inverted-Kumaraswamy distributed.

\cite{finkelstein2006mixture} addressed ordering mixture failure rates using ordered mixing distributions. \cite{navarro2008mean} studied ordering properties in coherent systems with dependent and identically distributed components, showing their results were independent of the common component distribution. \cite{navarro2013stochastic} proposed criteria to compare hazard rate and likelihood ratio orders in generalized mixtures. Similarly, \cite{balakrishnan2014stochastic} analyzed stochastic comparisons in series and parallel systems with independent, heterogeneous generalized exponential components. Building on this, \cite{navarro2016stochastic} introduced methods for comparing likelihood ratio and hazard rate orders in generalized mixtures. \cite{amini2017stochastic} extended these ideas by comparing stochastic orders in two FMMs, focusing on baseline RVs ordered as \(F_1 \geq_{hr} \ldots \geq_{hr} F_n\). \cite{navarro2017stochastic} provided necessary and sufficient conditions for stochastic comparisons in generalized distorted distributions, applicable to FMs with component-wise ordering such as \(F_1 \geq_{st} \ldots \geq_{st} F_n\). \cite{hazra2018stochastic} utilized multivariate chain majorization to compare two FMs under proportional reversed hazard rate, proportional hazard rate, and accelerated lifetime models. Further, \cite{barmalzan2021comparisons} studied finite \(\alpha\) MMs using various stochastic orders, while \cite{sattari2021stochastic} compared FMMs with generalized Lehmann-distributed components. Similarly, \cite{panja2021stochastic} analyzed FMMs under proportional odds, proportional reversed hazard rate, and proportional hazard rate models. \cite{barmalzan2022orderings} examined stochastic comparisons for the usual stochastic order and reversed hazard rate order, focusing on components from the location-scale family in two FMMs. \cite{bhakta2024stochasticcomparisons} extended these comparisons to reversed hazard rate, hazard rate, and usual stochastic orders for general parametric families. Recently, \cite{bhakta2024stochasticorderings} explored stochastic comparisons for two FMMs with inverted-Kumaraswamy-distributed components, using usual stochastic order, reversed hazard rate order, likelihood ratio order, and AFO in terms of reversed hazard rate order.

Let $\boldsymbol{X}=(X_1,\ldots,X_n)$ be a random vector with $n$ components, assuming the $i$th component is drawn from the $i$th sub-population. Assume that there are $n$ homogeneous infinite sub-populations of units and that $X_i$ denotes the lifetime of a unit in the $i$th sub-population, where $i\in\mathcal{I}_n$. The mixture of units drawn from these $n$ sub-populations is then represented by the RV $U_n$. Denote the SF, CDF and PDF of the $i$th RV $X_i$, $i\in\mathcal{I}_n$ by $\bar{F}_i(\cdot)$, $F_i(\cdot)$, and $f_i(\cdot)$, respectively. Now, the SF, CDF, and PDF of mixture RV $U_n$ are respectively as          
	
\begin{eqnarray}\label{eq1.1}
\bar{F}_{U_n}(x)=\sum_{i=1}^{n}r_i\bar{F}_i(x),~F_{U_n}(x)=\sum_{i=1}^{n}r_iF_i(x),~\mbox{and}~f_{U_n}(x)=\sum_{i=1}^{n}r_if_i(x),
\end{eqnarray}
where $\boldsymbol{r}=(r_1,\ldots,r_n)$ are the mixing proportions (weights) such as $\sum_{i=1}^{n}r_i=1$ and $r_i> 0$, for $i\in\mathcal{I}_n$. FMMs have several applications due to their capability to incorporate heterogeneity into studies. For instance, as discussed by \cite{panja2021stochastic} and \cite{hazra2018stochastic}, the lifetime distribution of a component operating in an environment with various stress levels (e.g., compression, voltage, and temperature) is a mixture distribution. This distribution represents a mixture of the underlying distributions of components under different stress levels. \cite{karlis2016finite} demonstrated that FMM is an effective method to address both heterogeneity and clustering in right-censored count data. Another practical example is that stock or asset returns are often modeled as a mixture of RVs (see \cite{kocuk2020incorporating}). For more details on FMMs, including various other applications, one may refer to \cite{amini2017stochastic}, \cite{everitt1981finite}, \cite{finkelstein2006mixture}, and \cite{franco2014generalized}.
	
In this paper, we have considered two FMMs for the ELS family of distributions. We refer to the baseline distributions as the ELS family of distributions throughout the study. A RV $X$ is said to follow ELS family of distributions if its CDF is given by   
\begin{eqnarray}\label{eq1.2}
F_X(x)\equiv F(x;\alpha,\sigma,\lambda)=F^{\alpha}\left(\frac{x-\sigma}{\lambda}\right),~x>\sigma+c\lambda,~\alpha>0,~\lambda>0,
\end{eqnarray}
where $c\in\mathbb{R}^+$, $\sigma\in\mathbb{R}^+$ is known as the location parameter and $\lambda$ is known as the scale parameter. In this case, $F(t)$ with $t\in(c,\infty)$ is the baseline CDF. For convenience, we denote $X\sim ELS(F;\alpha,\sigma,\lambda)$ if $X$ has the CDF given in (\ref{eq1.2}). Specifically, \cite{barmalzan2022orderings} explored LS distributed components within a FMM framework, where they considered the support of the baseline CDF $F(t)$ is as $t\in(0,\infty)$. However, this support is not suitable for the Pareto distribution because of its support is $t\geq1$. In this context, if we consider the Pareto distribution, then most of the established results in \cite{barmalzan2022orderings} are not valid within the support $0<t<1$. So, in contrast, the baseline CDF utilized in this study accommodates the Pareto distribution and other distributions that hold for $c=0$. Therefore, this flexibility in supporting distributions with a lower bound shift serves as the primary motivation for adopting the support $t\in(c, \infty)$. In addition, a FMM with ELS distributed components offers significant flexibility for modeling complex, skewed, and heavy-tailed data. It effectively captures multi-modal patterns and accommodates heterogeneous subpopulations. This model enhances robustness to outliers and improves clustering and classification tasks. In reliability analysis, it models diverse hazard functions, while in finance and insurance, it accurately assesses extreme risks. Its adaptability makes it valuable across various fields.
	
The established stochastic ordering results for ELS models on FMMs offer a novel contribution by extending the theoretical framework of stochastic comparisons to a broader class of flexible distributions. These ELS models generalize classical distributions by incorporating shape parameters through exponentiation, allowing for improved modeling of skewness and tail behavior. The research advances existing literature by rigorously identifying conditions under which various stochastic orderings (for instance, usual stochastic order, hazard rate order, and likelihood ratio order) are preserved between FMs of ELS models. This is significant because such results were previously limited to simpler or non-exponentiated distribution families.
	
The implications are substantial for applied probability and statistical modeling, especially in areas such as reliability engineering, risk assessment, survival analysis, and economics. For instance, the ordering results can be used to compare the reliability of complex systems modeled as mixtures of lifetime distributions or to assess risk in portfolios modeled by mixtures of income distributions. These findings enable practitioners to make informed decisions based on stochastic dominance, even when dealing with complex data distributions that require the flexibility of ELS models. Moreover, the results support improved model selection, hypothesis testing, and inference in scenarios where MMs are used to represent heterogeneous populations.
	
The paper's outline is as follows. In Section \ref{section2}, we present some essential notions, definitions, and key lemmas that will be used throughout the study. Section \ref{section3} provides sufficient conditions for comparing two single-outlier FMMs with ELS family distributed components in terms of the usual stochastic order, reversed hazard rate order, and likelihood ratio order. In Section \ref{section4}, stochastic comparison results between two multiple-outlier FMMs are established in the sense of the reversed hazard rate order, likelihood ratio order, and AFO in terms of the reversed hazard rate. Section \ref{section5} contains several illustrative examples and counterexamples, aiming to elucidate the theoretical results. Finally, Section \ref{section6} has some concluding remarks.
	
All the RVs are nonnegative and absolutely continuous.  The words ``increasing'' and ``decreasing'' respectively mean ``nondecreasing'' and ``nonincreasing''. We  write ``$a\stackrel{sign}{=}b$'' to indicate that the signs of $a$ and $b$ are the same. For a continuously differentiable function $\phi(t)$, the first order derivative with respect to $t$ is represented by $\phi^\prime(t)$.
	
\section{Preliminaries}\label{section2}
In this section, we provide some fundamental concepts and important lemmas which will be helpful in subsequent developments.
\begin{table}[h!]\label{table:2}
\centering
\begin{tabular}{llll}
\toprule
&~~~~~~~~~~~~~~~~~~~~~~~~~~~~~~~~\textbf{Notation}&&\\
\midrule
$n$ & Number of variables & $\mathbb{R}^n$ & $\left\lbrace \boldsymbol{a}=(a_1,\ldots,a_n):~a_i\in\mathbb{R},\forall~i\in\mathcal{I}_n\right\rbrace $\\
$\mathcal{I}_{n-1}$ & $\left\lbrace 1,\ldots,n-1\right\rbrace$ & $\mathcal{E}_n$ & $\left\lbrace\boldsymbol{u}=(u_1,\ldots,u_n)\in\mathbb{R}^n:~u_1\leq\ldots\leq u_n\right\rbrace$\\
$\mathcal{I}_{n}$ & $\left\lbrace 1,\ldots,n\right\rbrace$ & $\mathcal{E}_n^+$ & $\left\lbrace\boldsymbol{v}=(v_1,\ldots,v_n)\in\mathbb{R}^n:~0\leq v_1\leq\ldots\leq v_n\right\rbrace$\\
$\mathbb{R}$ & $\left(-\infty,+\infty\right)$ & $\mathcal{D}_n$ & $\left\lbrace \boldsymbol{u}=(u_1,\ldots,u_n)\in\mathbb{R}^n:~u_1\geq\ldots\geq u_n\right\rbrace$\\
$\mathbb{R}^+$ & $[0,+\infty)$ & $\mathcal{D}_n^+$ & $\left\lbrace \boldsymbol{v}=(v_1,\ldots,v_n)\in\mathbb{R}^n:~v_1\geq\ldots\geq v_n\geq 0\right\rbrace $\\
\bottomrule
\end{tabular}
\end{table}
Suppose $X$ and $Y$ are two absolutely continuous and nonnegative RVs with PDFs $f_X(\cdot)$ and $f_Y(\cdot)$, CDFs $F_X(\cdot)$ and $F_Y(\cdot)$, SFs $\bar{F}_X(\cdot)\equiv 1-F_X(\cdot)$ and $\bar{F}_Y(\cdot)\equiv 1-F_Y(\cdot)$, and RHRFs $\tilde{h}_X(\cdot)\equiv f_X(\cdot)/F_X(\cdot)$ and $\tilde{h}_Y(\cdot)\equiv f_Y(\cdot)/F_Y(\cdot)$, respectively. Below, we present definitions of various stochastic orders. 
    
\begin{definition}
A RV $X$ is said to be smaller than $Y$ in the sense of the
\begin{itemize}
\item usual stochastic order (denoted as $X\leq_{st}Y$) if $F_X(x)\geq F_Y(x)$, for all $x\in\mathbb{R}^+,$ or equivalently, if $\bar{F}_X(x)\leq\bar{F}_Y(x)$, for all $x\in \mathbb{R}^+$;  
\item reversed hazard rate order (denoted as $X\leq_{rh}Y$) if $F_Y(x)/F_X(x)$ is increasing in $x$, for all $x\in\mathbb{R}^+$, or equivalently, if $\tilde{h}_X(x)\leq \tilde{h}_Y(x)$, for all $x\in\mathbb{R}^+$;  
\item likelihood ratio order (denoted as $X\leq_{lr}Y$) if $f_Y(x)/f_X(x)$ is increasing in $x$, for all $x\in\mathbb{R}^+$.
\end{itemize}   
\end{definition}
	
\begin{definition}\label{def2.2}
A RV $X$ is said to be ageing faster than $Y$ in terms of the reversed hazard rate (denoted by $X\leq_{R-rh}Y$), if $\tilde{h}_X(x)/\tilde{h}_Y(x)$ is increasing in $x\in\mathbb{R^+}$. 
\end{definition}
	
It is known that $X\leq_{lr}Y\Rightarrow X\leq_{rh}Y\Rightarrow X\leq_{st}Y$. See \cite{shaked2007stochastic}, \cite{li2013stochastic}, and \cite{muller2002comparison} for further information on stochastic orders and their applications. We introduce the idea of majorization order below. Please see \cite{marshall2011inequalities} for more information.

\begin{definition}
Let ${\boldsymbol{x}}=(x_1,\ldots,x_n)\in\mathbb{R}^n$ and ${\boldsymbol{y}}=(y_1,\ldots,y_n)\in\mathbb{R}^n$ be two vectors, with $x_{(1)}\leq\ldots\leq x_{(n)}$ representing the increasing arrangements of the components of $\boldsymbol{x}$. A vector $\boldsymbol{x}$ is supposed to be majorized via the vector $\boldsymbol{y}$ (denoted as ${\boldsymbol{x}}\stackrel{m}{\preccurlyeq}{\boldsymbol{y}}$) if
$\sum_{i=1}^{j}x_{(i)}\geq\sum_{i=1}^{j} y_{(i)}$, for all $j\in\mathcal{I}_{n-1}$ and $\sum_{i=1}^{n}x_{(i)}=\sum_{i=1}^{n}y_{(i)}$.
\end{definition}			
	
The following definition describes that the concept of majorization preserves Schur-convexity of a function. 
	
\begin{definition}\label{def2.4}
Let $\mathscr{A}\subseteq\mathbb{R}$. A function $\phi:\mathscr{A}^n\rightarrow\mathbb{R}$ is said to be Schur-convex (Schur-concave) on $\mathscr{A}^n$ if
\begin{eqnarray*}
\boldsymbol{x}\stackrel{m}{\preccurlyeq}\boldsymbol{y} \Rightarrow\phi(\boldsymbol{x})\leq(\geq)~\phi(\boldsymbol{y}),~\mbox{for all}~ \boldsymbol{x},\boldsymbol{y}\in\mathscr{A}^n.	
\end{eqnarray*}
\end{definition}
	
Now, let us introduce the subsequent lemmas, which will be utilized to obtain the main results.
	
\begin{lemma}\label{lemma2.3}
(\cite{marshall2011inequalities}, Theorem $A.3$, p. $83$) Consider the function $\zeta:\mathcal{D}_n\rightarrow\mathbb{R}$ is continuous on $\mathcal{D}_n$ and continuously differentiable on the interior of $\mathcal{D}_n$. Then, $\zeta$ is Schur-convex (Schur-concave) on $\mathcal{D}_n$, if and only if $\zeta_{(k)}(\boldsymbol{x})$ is decreasing (increasing) in $k\in\mathcal{I}_n$, for all $\boldsymbol{x}$ in the interior of $\mathcal{D}_n$, where  $\zeta_{(k)}(\boldsymbol{x})=\partial\zeta(\boldsymbol{x})/\partial x_k$.
\end{lemma}

The following lemma can be established using arguments similar to those in Theorem $A.3$ of \cite{marshall2011inequalities} (see p. $83$).
	
\begin{lemma}\label{lemma2.4}
Suppose the function $\zeta:\mathcal{E}_n\rightarrow\mathbb{R}$ is continuous on $\mathcal{E}_n$ and continuously differentiable on the interior of $\mathcal{E}_n$. Then, $\zeta$ is Schur-convex (Schur-concave) on $\mathcal{D}_n$, if and only if $\zeta_{(k)}(\boldsymbol{x})$ is increasing (decreasing) in $k\in\mathcal{I}_n$, for all $\boldsymbol{x}$ in the interior of $\mathcal{E}_n$.
\end{lemma}
	
\begin{remark}\label{remark2.1}
The result present in Lemma \ref{lemma2.3} and \ref{lemma2.4} also hold if the spaces $\mathcal{D}_n^+$ and $\mathcal{E}_n^+$ are used in place of $\mathcal{D}_n$ and $\mathcal{E}_n$, respectively.
\end{remark}
	
\section{Stochastic comparisons between two single-outlier FMMs}\label{section3}
Here, we consider two different FMMs with general ELS distributed components and obtain some ordering results between them. We denote the PDF and CDF of the $i$-th subpopulation by $f(x;\alpha_i,\sigma_i,\lambda_i)$ and $F(x;\alpha_i,\sigma_i,\lambda_i)$, respectively, where $i\in\mathcal{I}_{n}.$ Further, let $U_n(\boldsymbol{r};\boldsymbol{\alpha},\boldsymbol{\sigma},\boldsymbol{\lambda})$ be the MRV, constructed from these subpopulations, where $\boldsymbol{r}=(r_1,\ldots,r_n)$, $\boldsymbol{\alpha}=(\alpha_1,\ldots,\alpha_n)$, $\boldsymbol{\sigma}=(\sigma_1,\ldots,\sigma_n)$, and $\boldsymbol{\lambda}=(\lambda_1,\ldots,\lambda_n)$. The CDF of $U_n(\boldsymbol{r};\boldsymbol{\alpha},\boldsymbol{\sigma},\boldsymbol{\lambda})$ is written as
\begin{eqnarray}\label{eq3.3}
F_{U_n(\boldsymbol{r};\boldsymbol{\alpha},\boldsymbol{\sigma},\boldsymbol{\lambda})}(x)=\sum_{i=1}^{n}r_iF^{\alpha_i}\left(\frac{x-\sigma_i}{\lambda_i}\right)I(x>\sigma_i+c\lambda_i),~~c\in\mathbb{R}^+,~\alpha_i,\lambda_i>0,~\sigma_i\in\mathbb{R},
\end{eqnarray}
where $F(\cdot)$ is the baseline CDF and $I(x>\sigma_i+c\lambda_i)$ is an indicator function of the following form:
\begin{eqnarray*}
I(x>\sigma_i+c\lambda_i)=
\begin{cases}
1 & \text{if}~x>\sigma_i+c\lambda_i;\\
0 & \text{if}~x\leq\sigma_i+c\lambda_i,\\
\end{cases}
\end{eqnarray*}
for $i\in\mathcal{I}_n$. In the following proposition, we establish that the PDF obtained after differentiating (\ref{eq3.3}) with respect to $x$ is a proper PDF of the MRV $U_n(\boldsymbol{r};\boldsymbol{\alpha},\boldsymbol{\sigma},\boldsymbol{\lambda})$.
	
\begin{proposition}
Let $\sum_{i=1}^{n}r_i=1$. The function
\begin{eqnarray}\label{eq2.1}
f_{U_n(\boldsymbol{r};\boldsymbol{\alpha},\boldsymbol{\sigma},\boldsymbol{\lambda})}(x)&=&\sum_{i=1}^{n}\frac{r_i\alpha_i}{\lambda_i}F^{\alpha_i-1}\left(\frac{x-\sigma_i}{\lambda_i}\right)f\left(\frac{x-\sigma_i}{\lambda_i}\right)I(x>\sigma_i+c\lambda_i)
\end{eqnarray}
is a proper PDF, where $f(\cdot)$ is the baseline PDF of the nonnegative RV with unbounded support.
\end{proposition}
	
\begin{proof}
It is obvious that $f_{U_n(\boldsymbol{r};\boldsymbol{\alpha},\boldsymbol{\sigma},\boldsymbol{\lambda})}(x)\geq0.$ Further,
\begin{eqnarray*}
\int_{-\infty}^{+\infty}f_{U_n(\boldsymbol{r};\boldsymbol{\alpha},\boldsymbol{\sigma},\boldsymbol{\lambda})}(x)dx&=\displaystyle\sum_{i=1}^{n}\int_{\sigma_i+c\lambda_i}^{\infty}\frac{r_i\alpha_i}{\lambda_i}F^{\alpha_i-1}\left(\frac{x-\sigma_i}{\lambda_i}\right)f\left(\frac{x-\sigma_i}{\lambda_i}\right)dx.
\end{eqnarray*}
Let $F\left(\frac{x-\sigma_i}{\lambda_i}\right)=z_i,$ for $i\in\mathcal{I}_n$.
Thus, $\frac{1}{\lambda_i}f\left(\frac{x-\sigma_i}{\lambda_i}\right)dx=dz_i$. When $x=\sigma_i+c\lambda_i$ and $x=+\infty$ then $z_i=F(c)=0$ and $z_i=F(+\infty)=1$, respectively. Then,
\begin{eqnarray*}
\int_{-\infty}^{+\infty}f_{U_n(\boldsymbol{r};\boldsymbol{\alpha},\boldsymbol{\sigma},\boldsymbol{\lambda})}(x)dx=\sum_{i=1}^{n}\int_{0}^{1}r_i\alpha_i z_i^{\alpha_i-1}dz_i
=\sum_{i=1}^{n}r_i
=1.
\end{eqnarray*}
Thus, $f_{U_n(\boldsymbol{r};\boldsymbol{\alpha},\boldsymbol{\sigma},\boldsymbol{\lambda})}(x)$ is a proper PDF.
\end{proof}
	
In the first theorem, the usual stochastic order is established for two FMMs of heterogeneous general ELS distributed components, assuming that the distributional parameter vectors are connected by inequalities. Here, we have considered a common mixing proportion parameter vector for two FMMs.

\begin{theorem}\label{theorem3.1}
Let $\bar{F}_{U_n(\boldsymbol{r};\boldsymbol{\alpha},\boldsymbol{\sigma},\boldsymbol{\lambda})}(x)=1-\sum_{i=1}^{n}r_iF^{\alpha_i}\left(\frac{x-\sigma_i}{\lambda_i}\right)I(x>\sigma_i+c\lambda_i)$ and
$\bar{F}_{U_n(\boldsymbol{r};\boldsymbol{\beta},\boldsymbol{\mu},\boldsymbol{\theta})}(x)=1-\sum_{i=1}^{n}r_iF^{\beta_i}(\frac{x-\mu_i}{\theta_i})I(x>\mu_i+c\theta_i)$ be the SFs of two MRVs $U_n(\boldsymbol{r};\boldsymbol{\alpha},\boldsymbol{\sigma},\boldsymbol{\lambda})$ and $U_n(\boldsymbol{r};\boldsymbol{\beta},\boldsymbol{\mu},\boldsymbol{\theta})$, respectively. Assume that $\boldsymbol{\sigma},\boldsymbol{\lambda},\boldsymbol{\mu},\boldsymbol{\theta}\in\mathcal{E}_{n}^{+}$. Then, we have 
\begin{eqnarray*}
U_n(\boldsymbol{r};\boldsymbol{\alpha},\boldsymbol{\sigma},\boldsymbol{\lambda})\leq_{st}U_n(\boldsymbol{r};\boldsymbol{\beta},\boldsymbol{\mu},\boldsymbol{\theta}),
\end{eqnarray*}
provided $\alpha_i\leq\beta_i$, $\sigma_i\leq\mu_i$, and $\lambda_i\leq\theta_i$, for $i\in\mathcal{I}_{n}.$
\end{theorem}
	
\begin{proof}
To obtain the desired result, it is required to show that
$\bar{F}_{U_n(\boldsymbol{r};\boldsymbol{\alpha},\boldsymbol{\sigma},\boldsymbol{\lambda})}(x)\leq\bar{F}_{U_n(\boldsymbol{r};\boldsymbol{\beta},\boldsymbol{\mu},\boldsymbol{\theta})}(x).$
Under the assumptions $\boldsymbol{\sigma},\boldsymbol{\lambda},\boldsymbol{\mu},\boldsymbol{\theta}\in\mathcal{E}_{n}^{+}$, the SF of $U_n(\boldsymbol{r};\boldsymbol{\alpha},\boldsymbol{\sigma},\boldsymbol{\lambda})$ is written as  
\[\bar{F}_{U_n(\boldsymbol{r};\boldsymbol{\alpha},\boldsymbol{\sigma},\boldsymbol{\lambda})}(x)=
\begin{cases}
1 &;x\leq\sigma_1+c\lambda_1\\
1-r_1F^{\alpha_1}\left(\frac{x-\sigma_1}{\lambda_1}\right)  &;\sigma_1+c\lambda_1<x\leq\sigma_2+c\lambda_2\\~~~~~~~~~~~~\vdots\\
1-\sum_{i=1}^{k}r_iF^{\alpha_i}\left(\frac{x-\sigma_i}{\lambda_i}\right)&;\sigma_k+c\lambda_k<x\leq\sigma_{k+1}+c\lambda_{k+1}\\~~~~~~~~~~~~\vdots\\
1-\sum_{i=1}^{n}r_iF^{\alpha_i}\left(\frac{x-\sigma_i}{\lambda_i}\right)&;x>\sigma_{n}+c\lambda_{n},
\end{cases}\]
where $k\in\mathcal{I}_{n-1}$. The SF of $U_n(\boldsymbol{r};\boldsymbol{\beta},\boldsymbol{\mu},\boldsymbol{\theta})$ can be written similarly. From the assumption, we have $\lambda_i\leq\theta_i$ and $\sigma_i\leq\mu_i$, implying that $\sigma_i+c\lambda_i\leq\mu_i+c\theta_i$, $\forall~i\in\mathcal{I}_n$. Further, using $\lambda_i\leq\theta_i$, $\sigma_i\leq\mu_i$, and $\alpha_i\leq\beta_i$, we obtain
$F^{\alpha_i}(\frac{x-\sigma_i}{\lambda_i})\geq F^{\beta_i}(\frac{x-\mu_i}{\theta_i})$, from which it is clear that 
\begin{eqnarray}\label{eq3.7}
\sum_{i=1}^{k}r_iF^{\alpha_i}\left(\frac{x-\sigma_i}{\lambda_i}\right)\geq \sum_{i=1}^{l}r_iF^{\beta_i}\left(\frac{x-\mu_i}{\theta_i}\right),
\end{eqnarray}
for $k\geq l;$ $k,l\in\mathcal{I}_{n}$. Under the given settings, the value of $x$ may fall in different subintervals, which have been considered in the following cases to prove the desired inequality between the SFs of $U_n(\boldsymbol{r};\boldsymbol{\alpha},\boldsymbol{\sigma},\boldsymbol{\lambda})$ and $U_n(\boldsymbol{r};\boldsymbol{\beta},\boldsymbol{\mu},\boldsymbol{\theta})$:\\
$\boldsymbol{{Case~I}}:$ Let $x\leq\sigma_1+c\lambda_1$. Then, it is obvious that 
$\bar{F}_{U_n(\boldsymbol{r};\boldsymbol{\alpha},\boldsymbol{\sigma},\boldsymbol{\lambda})}(x)=1=\bar{F}_{U_n(\boldsymbol{r};\boldsymbol{\beta},\boldsymbol{\mu},\boldsymbol{\theta})}(x).$\\
$\boldsymbol{{Case~II}}:$ Consider $\sigma_k+c\lambda_k<x\leq\ \sigma_{k+1}+c\lambda_{k+1}$, for $k\in\mathcal{I}_{n-1}$. Using (\ref{eq3.7}), it can be established that 
\begin{eqnarray*}
\bar{F}_{U_n(\boldsymbol{r};\boldsymbol{\alpha},\boldsymbol{\sigma},\boldsymbol{\lambda})}(x)=1-\sum_{i=1}^{k}r_iF^{\alpha_i}\left(\frac{x-\sigma_i}{\lambda_i}\right)\leq 1-\sum_{i=1}^{l}r_iF^{\beta_i}\left(\frac{x-\mu_i}{\theta_i}\right)=\bar{F}_{U_n(\boldsymbol{r};\boldsymbol{\beta},\boldsymbol{\mu},\boldsymbol{\theta})}(x),
\end{eqnarray*}
for $k\geq l;~k,l\in\mathcal{I}_{n-1}.$\\
$\boldsymbol{{Case~III}}:$ Suppose $\mu_l+c\theta_l<x\leq\mu_{l+1}+c\theta_{l+1}$. Then, by using (\ref{eq3.7}), we obtain
\begin{eqnarray*}
\bar{F}_{U_n(\boldsymbol{r};\boldsymbol{\alpha},\boldsymbol{\sigma},\boldsymbol{\lambda})}(x)=1-\sum_{i=1}^{k}r_iF^{\alpha_i}\left(\frac{x-\sigma_i}{\lambda_i}\right)\leq 1-\sum_{i=1}^{l}r_iF^{\beta_i}\left(\frac{x-\mu_i}{\theta_i}\right)=\bar{F}_{U_n(\boldsymbol{r};\boldsymbol{\beta},\boldsymbol{\mu},\boldsymbol{\theta})}(x),
\end{eqnarray*}
for $k\geq l$; $k,l\in\mathcal{I}_{n-1}.$\\
$\boldsymbol{{Case~IV}}:$ Assume that $\sigma_k+c\lambda_k<x\leq\mu_{l+1}+c\theta_{l+1}$. Now, for $k\geq l,$ where $k,l\in\mathcal{I}_{n-1},$ utilizing (\ref{eq3.7}), we obtain

\begin{eqnarray*}
\bar{F}_{U_n(\boldsymbol{r};\boldsymbol{\alpha},\boldsymbol{\sigma},\boldsymbol{\lambda})}(x)=1-\sum_{i=1}^{k}r_iF^{\alpha_i}\left(\frac{x-\sigma_i}{\lambda_i}\right)\leq 1-\sum_{i=1}^{l}r_iF^{\beta_i}\left(\frac{x-\mu_i}{\theta_i}\right)=\bar{F}_{U_n(\boldsymbol{r};\boldsymbol{\beta},\boldsymbol{\mu},\boldsymbol{\theta})}(x).
\end{eqnarray*}
$\boldsymbol{{Case~V}}:$ Let $\mu_{l}+c\theta_{l}<x\leq\sigma_{k+1}+c\lambda_{k+1}$. Under this case, using (\ref{eq3.7}), we get
\begin{eqnarray*}
\bar{F}_{U_n(\boldsymbol{r};\boldsymbol{\alpha},\boldsymbol{\sigma},\boldsymbol{\lambda})}(x)=1-\sum_{i=1}^{k}r_iF^{\alpha_i}\left(\frac{x-\sigma_i}{\lambda_i}\right)\leq 1-\sum_{i=1}^{l}r_iF^{\beta_i}\left(\frac{x-\mu_i}{\theta_i}\right)=\bar{F}_{U_n(\boldsymbol{r};\boldsymbol{\beta},\boldsymbol{\mu},\boldsymbol{\theta})}(x),
\end{eqnarray*}
for $k\geq l;~k,l\in\mathcal{I}_{n-1}.$\\
$\boldsymbol{{Case~VI}}:$ Finally, when $x>\mu_n+c\theta_n$. In this case, under the assumptions made, we clearly have   
\begin{eqnarray*}
\bar{F}_{U_n(\boldsymbol{r};\boldsymbol{\alpha},\boldsymbol{\sigma},\boldsymbol{\lambda})}(x)=1-\sum_{i=1}^{n}r_iF^{\alpha_i}\left(\frac{x-\sigma_i}{\lambda_i}\right)\leq 1-\sum_{i=1}^{n}r_iF^{\beta_i}\left(\frac{x-\mu_i}{\theta_i}\right)=\bar{F}_{U_n(\boldsymbol{r};\boldsymbol{\beta},\boldsymbol{\mu},\boldsymbol{\theta})}(x).
\end{eqnarray*}
Combining the results obtained in the above six cases, the usual stochastic order between $U_n(\boldsymbol{r};\boldsymbol{\alpha},\boldsymbol{\sigma},\boldsymbol{\lambda})$ and $U_n(\boldsymbol{r};\boldsymbol{\beta},\boldsymbol{\mu},\boldsymbol{\theta})$ readily follows. 
\end{proof}

In order to validate the established result in Theorem \ref{theorem3.1}, we present Example \ref{example4.1} in Section \ref{section4}. Further, we provide a counterexample (see Counterexample \ref{Coun4.1}) to show that if the inequalities between all the model parameters are not satisfied, then the usual stochastic order between two FMs may not hold. 

\begin{remark}
The established result in Theorem \ref{theorem3.1} also holds if $\boldsymbol{\sigma},\boldsymbol{\lambda},\boldsymbol{\mu},\boldsymbol{\theta}\in\mathcal{D}_{n}^{+}$.
\end{remark} 
	
\begin{remark}
In Theorem \ref{theorem3.1}, we have considered a common mixing proportion parameter vector between the FMMs. In this regard, it is natural to raise a question that ``does the usual stochastic ordering in Theorem \ref{theorem3.1} hold if we do not assume a common mixing proportion parameter vector between the FMMs?'' In this purpose, we have considered several numerical examples (not reported here for the sake of brevity) by taking $r_i\ge s_i$, for $i\in\mathcal{I}_{n-1}$, which provide an evidence of the existence of the similar usual stochastic ordering result in Theorem \ref{theorem3.1}. However, we are not able to establish it theoretically. Thus, this problem when the mixing proportion parameter vectors are not common can be considered as an open problem. 
\end{remark}
	
In the previous theorem, we have proved the usual stochastic order between two MRVs with a common mixing proportion parameter vector. Thus, it is of interest to study if the usual stochastic order can be extended to other stronger stochastic orders, say reversed hazard rate order or likelihood ratio order. It is observed in Counterexample \ref{Coun4.2} and Counterexample \ref{Coun4.22} that under the same settings as in Theorem \ref{theorem3.1}, this extension is not possible. In the following consecutive theorems, we have proposed sufficient conditions (other than that of Theorem \ref{theorem3.1}), under which the extension of stochastic ordering result from usual stochastic order to reversed hazard rate order and likelihood ratio order is possible. In the next result, we assume different model parameter vectors as well as different mixing proportion vectors for two FMMs. Denote $m_1=\displaystyle\min_{i\in\mathcal{I}_{n}}\{\sigma_i+c\lambda_i\}$ and $m_2=\displaystyle\min_{i\in\mathcal{I}_{n}}\{\mu_i+c\theta_i\}$.

\begin{theorem}\label{theorem3.2}
Let $F_{U_n(\boldsymbol{r};\boldsymbol{\alpha},\boldsymbol{\sigma},\boldsymbol{\lambda})}(x)=\sum_{i=1}^{n}r_iF^{\alpha_i}(\frac{x-\sigma_i}{\lambda_i})I(x>\sigma_i+c\lambda_i)$ and $F_{U_n(\boldsymbol{s};\boldsymbol{\beta},\boldsymbol{\mu},\boldsymbol{\theta})}(x)=\sum_{i=1}^{n}s_i\\F^{\beta_i}(\frac{x-\mu_i}{\theta_i})I(x>\mu_i+c\theta_i)$ be the CDFs of two MRVs $U_n(\boldsymbol{r};\boldsymbol{\alpha},\boldsymbol{\sigma},\boldsymbol{\lambda})$ and $U_n(\boldsymbol{s};\boldsymbol{\beta},\boldsymbol{\mu},\boldsymbol{\theta})$, respectively. Further, assume that $\max\left\lbrace\alpha_1,\ldots,\alpha_n\right\rbrace\leq \min\left\lbrace\beta_1,\ldots,\beta_n\right\rbrace$, $\max\left\lbrace\sigma_1,\ldots,\sigma_n\right\rbrace\leq \min\left\lbrace\mu_1,\ldots,\mu_n\right\rbrace$, and $\max\left\lbrace\lambda_1,\ldots,\lambda_n\right\rbrace\leq \min\left\lbrace\theta_1,\ldots,\theta_n\right\rbrace$. Then, for $x>m_1$, we have
\begin{eqnarray*}
U_n(\boldsymbol{r};\boldsymbol{\alpha},\boldsymbol{\sigma},\boldsymbol{\lambda})\leq_{rh}U_n(\boldsymbol{s};\boldsymbol{\beta},\boldsymbol{\mu},\boldsymbol{\theta}),
\end{eqnarray*}
provided $t\tilde{h}(t)=t\frac{f(t)}{F(t)}$ is decreasing in $t>c\geq0$.
\end{theorem}

\begin{proof}
To obtain the desired result, it is necessary to prove that $\psi(x)=F_{U_n(\boldsymbol{s};\boldsymbol{\beta},\boldsymbol{\mu},\boldsymbol{\theta})}(x)/F_{U_n(\boldsymbol{r};\boldsymbol{\alpha},\boldsymbol{\sigma},\boldsymbol{\lambda})}(x)$ is increasing in $x>m_1$. Differentiating $\psi(x)$ with respect to $x$, we obtain
\allowdisplaybreaks{\begin{align}\label{eq3.12}
\psi^{\prime}(x) &\stackrel{sign}{=}F_{U_n(\boldsymbol{r};\boldsymbol{\alpha},\boldsymbol{\sigma},\boldsymbol{\lambda})}(x)f_{U_n(\boldsymbol{s};\boldsymbol{\beta},\boldsymbol{\mu},\boldsymbol{\theta})}(x)-f_{U_n(\boldsymbol{r};\boldsymbol{\alpha},\boldsymbol{\sigma},\boldsymbol{\lambda})}(x)F_{U_n(\boldsymbol{s};\boldsymbol{\beta},\boldsymbol{\mu},\boldsymbol{\theta})}(x)\nonumber\\
&~=\sum_{i=1}^{n}r_iF^{\alpha_i}\left(\frac{x-\sigma_i}{\lambda_i}\right)I(x>\sigma_i+c\lambda_i)\sum_{i=1}^{n}\frac{s_i\beta_i}{\theta_i}F^{\beta_i-1}\left(\frac{x-\mu_i}{\theta_i}\right)f\left(\frac{x-\mu_i}{\theta_i}\right)\nonumber\\
&~~~~\times I(x>\mu_i+c\theta_i)-\sum_{i=1}^{n}\frac{r_i\alpha_i}{\lambda_i}F^{\alpha_i-1}\left(\frac{x-\sigma_i}{\lambda_i}\right)f\left(\frac{x-\sigma_i}{\lambda_i}\right)I(x>\sigma_i+c\lambda_i)\nonumber\\
&~~~~\times\sum_{i=1}^{n}s_iF^{\beta_i}\left(\frac{x-\mu_i}{\theta_i}\right)I(x>\mu_i+c\theta_i)\nonumber\\
&=\sum_{i=1}^{n}\sum_{j=1}^{n}r_is_jF^{\alpha_i}\left(\frac{x-\sigma_i}{\lambda_i}\right)F^{\beta_j}\left(\frac{x-\mu_j}{\theta_j}\right)I(x>\sigma_i+c\lambda_i)I(x>\mu_j+c\theta_j)\nonumber\\
&~~~~\times\left[\frac{\beta_j}{\theta_j}\tilde{h}\left(\frac{x-\mu_j}{\theta_j}\right)-\frac{\alpha_i}{\lambda_i}\tilde{h}\left(\frac{x-\sigma_i}{\lambda_i}\right)\right].
\end{align}}
In addition, we assume that $\max\left\lbrace\alpha_1,\ldots,\alpha_n\right\rbrace\leq \min\left\lbrace\beta_1,\ldots,\beta_n\right\rbrace$, $\max\left\lbrace\sigma_1,\ldots,\sigma_n\right\rbrace\leq \min\left\lbrace\mu_1,\ldots,\mu_n\right\rbrace$, $\max\left\lbrace\lambda_1,\ldots,\lambda_n\right\rbrace\leq \min\left\lbrace\theta_1,\ldots,\theta_n\right\rbrace$, and $t\tilde{h}(t)$ is decreasing in $t>c$. Thus,
\begin{eqnarray}\label{eq3.13}
\frac{\alpha_i}{x-\sigma_i}\frac{x-\sigma_i}{\lambda_i}\tilde{h}\left(\frac{x-\sigma_i}{\lambda_i}\right)\leq\frac{\beta_j}{x-\mu_j}\frac{x-\mu_j}{\theta_j}\tilde{h}\left(\frac{x-\mu_j}{\theta_j}\right),
\end{eqnarray}
for $i,j\in\mathcal{I}_n$. Now, substituting (\ref{eq3.13}) in (\ref{eq3.12}), it is easy to see that $\psi^{\prime}(x)\geq0$, which means that $\psi(x)$ is increasing in $x>m_1$. Hence, the theorem is proved. 
\end{proof}
	
In the next result, the likelihood ratio order between two MRVs $U_n(\boldsymbol{r};\boldsymbol{\alpha},\sigma,\lambda)$ and $U_n(\boldsymbol{s};\boldsymbol{\beta},\sigma,\lambda)$ has been established. Here, we have considered heterogeneity in the mixing proportion parameter vectors and shape parameter vectors. The location parameter $\sigma$ and scale parameter $\lambda$ are assumed to be fixed and common for both MMs.    
	
\begin{theorem}
\label{theorem3.3}
Let $F_{U_n(\boldsymbol{r};\boldsymbol{\alpha},\sigma,\lambda)}(x)=\sum_{i=1}^{n}r_iF^{\alpha_i}(\frac{x-\sigma}{\lambda})$ and $F_{U_n(\boldsymbol{s};\boldsymbol{\beta},\sigma,\lambda)}(x)=\sum_{i=1}^{n}s_iF^{\beta_i}(\frac{x-\sigma}{\lambda})$ be the CDFs of two MRVs $U_n(\boldsymbol{r};\boldsymbol{\alpha},\sigma,\lambda)$ and $U_n(\boldsymbol{s};\boldsymbol{\beta},\sigma,\lambda)$, respectively, for $x>\sigma+c\lambda$. Furthermore, assume that $\max\left\lbrace\alpha_1,\ldots,\alpha_n\right\rbrace\leq \min\left\lbrace\beta_1,\ldots,\beta_n\right\rbrace$. Then, for $x>\sigma+c\lambda$, we have
\begin{eqnarray*}
U_n(\boldsymbol{r};\boldsymbol{\alpha},\sigma,\lambda)\leq_{lr}U_n(\boldsymbol{s};\boldsymbol{\beta},\sigma,\lambda).
\end{eqnarray*}
\end{theorem}

\begin{proof}
To obtain the necessary result, it suffices to prove that $\xi(x)=\frac{f_{U_n(\boldsymbol{s};\boldsymbol{\beta},\sigma,\lambda)}(x)}{f_{U_n(\boldsymbol{r};\boldsymbol{\alpha},\sigma,\lambda)}(x)}$ is increasing in $x>\sigma+c\lambda$. In doing so, we have
\allowdisplaybreaks{\begin{align}\label{eq3.14}
\xi^\prime(x) &\stackrel{sign}{=}f_{U_n(\boldsymbol{r};\boldsymbol{\alpha},\sigma,\lambda)}(x)f^\prime_{U_n(\boldsymbol{s};\boldsymbol{\beta},\sigma,\lambda)}(x)-f^\prime_{U_n(\boldsymbol{r};\boldsymbol{\alpha},\sigma,\lambda)}(x)f_{U_n(\boldsymbol{s};\boldsymbol{\beta},\sigma,\lambda)}(x)\nonumber\\
&~=\sum_{i=1}^{n}\frac{r_i\alpha_i}{\lambda^2}F^{\alpha_i-1}\left(\frac{x-\sigma}{\lambda}\right)f\left(\frac{x-\sigma}{\lambda}\right)\sum_{i=1}^{n}\frac{s_i\beta_i}{\lambda}\Big[(\beta_i-1)F^{\beta_i-2}\left(\frac{x-\sigma}{\lambda}\right)f^2\left(\frac{x-\sigma}{\lambda}\right)\nonumber\\
&~~~~+F^{\beta_i-1}\left(\frac{x-\sigma}{\lambda}\right)f^\prime\left(\frac{x-\sigma}{\lambda}\right)\Big]-\sum_{i=1}^{n}\frac{r_i\alpha_i}{\lambda^2}\Big[(\alpha_i-1)F^{\alpha_i-2}\left(\frac{x-\sigma}{\lambda}\right)f^2\left(\frac{x-\sigma}{\lambda}\right)\nonumber\\
&~~~~+F^{\alpha_i-1}\left(\frac{x-\sigma}{\lambda}\right)f^\prime\left(\frac{x-\sigma}{\lambda}\right) \Big]\sum_{i=1}^{n}\frac{s_i\beta_i}{\lambda}F^{\beta_i-1}\left(\frac{x-\sigma}{\lambda}\right)f\left(\frac{x-\sigma}{\lambda}\right)\nonumber\\
&=\sum_{i=1}^{n}\sum_{j=1}^{n}\frac{r_i s_j\alpha_i\beta_j}{\lambda^3}F^{\alpha_i-1}\left(\frac{x-\sigma}{\lambda}\right)F^{\beta_j-1}\left(\frac{x-\sigma}{\lambda}\right)f^2\left(\frac{x-\sigma}{\lambda}\right)\Bigg[(\beta_j-1)\frac{f\left(\frac{x-\sigma}{\lambda}\right)}{F\left(\frac{x-\sigma}{\lambda}\right)}\nonumber\\
&~~~~+\frac{f^\prime\left(\frac{x-\sigma}{\lambda}\right)}{f\left(\frac{x-\sigma}{\lambda}\right)}-(\alpha_i-1)\frac{f\left(\frac{x-\sigma}{\lambda}\right)}{F\left(\frac{x-\sigma}{\lambda}\right)}-\frac{f^\prime\left(\frac{x-\sigma}{\lambda}\right)}{f\left(\frac{x-\sigma}{\lambda}\right)}\Bigg]\nonumber\\
&=\sum_{i=1}^{n}\sum_{j=1}^{n}\frac{r_i s_j\alpha_i\beta_j}{\lambda^3}F^{\alpha_i+\beta_j-3}\left(\frac{x-\sigma}{\lambda}\right)f^3\left(\frac{x-\sigma}{\lambda}\right)(\beta_j-\alpha_i).  
\end{align}}
From the assumption, $\max\left\lbrace\alpha_1,\ldots,\alpha_n\right\rbrace\leq \min\left\lbrace\beta_1,\ldots,\beta_n\right\rbrace$, equivalently $\beta_j\geq\alpha_i$,
for $i,j\in\mathcal{I}_n$. Thus, by using $\beta_j\geq\alpha_i$, from (\ref{eq3.14}), we obtain that $\xi^\prime(x)$ is nonnegative. This proves the theorem. 
\end{proof}

In Theorem \ref{theorem3.1}, we derived the usual stochastic ordering between two MRVs under the assumption that the associated MMs share a common vector of mixing proportion parameters. Here, we relax this assumption by allowing the mixing proportion vectors to differ. The sufficient conditions in this case are expressed in terms of the majorization relationship between the vectors of mixing proportions and shape parameters.	

	
\begin{theorem}\label{theorem3.4}
Let $F_{U_n(\boldsymbol{r};\boldsymbol{\alpha},\sigma,\lambda)}(x)=\sum_{i=1}^{n}r_iF^{\alpha_i}(\frac{x-\sigma}{\lambda})I(x>\sigma+c\lambda)$ and $F_{U_n(\boldsymbol{s};\boldsymbol{\beta},\mu,\theta)}(x)=\sum_{i=1}^{n}s_i\\F^{\beta_i}(\frac{x-\mu}{\theta})I(x>\mu+c\theta)$ be the CDFs of two MRVs $U_n(\boldsymbol{r};\boldsymbol{\alpha},\sigma,\lambda)$ and $U_n(\boldsymbol{s};\boldsymbol{\beta},\mu,\theta)$, respectively. Assume that $\boldsymbol{r},\boldsymbol{s}\in\mathcal{D}_n^+$ and $\boldsymbol{\alpha},\boldsymbol{\beta}\in\mathcal{E}_n^+$. Further, consider $\boldsymbol{r}\stackrel{m}\succcurlyeq\boldsymbol{s}$,  $\boldsymbol{\alpha}\stackrel{m}\succcurlyeq\boldsymbol{\beta}$, $\sigma\leq\mu$, and $\lambda\leq\theta$. Then, for $x>\sigma+c\lambda$, we have
\begin{eqnarray*}
U_n(\boldsymbol{r};\boldsymbol{\alpha},\sigma,\lambda)\leq_{st}U_n(\boldsymbol{s};\boldsymbol{\beta},\mu,\theta).
\end{eqnarray*}
\end{theorem}
	
\begin{proof}
To establish the desired result, it is necessary to show that	 $F_{U_n(\boldsymbol{r};\boldsymbol{\alpha},\sigma,\boldsymbol{\lambda})}(x)\geq F_{U_n(\boldsymbol{s};\boldsymbol{\beta},\mu,\theta)}(x)$, which can be attained by
\begin{eqnarray}\label{eq-3.25}
F_{U_n(\boldsymbol{r};\boldsymbol{\alpha},\sigma,\lambda)}(x)\geq F_{U_n(\boldsymbol{s};\boldsymbol{\alpha},\sigma,\lambda)}(x)\geq F_{U_n(\boldsymbol{s};\boldsymbol{\beta},\sigma,\lambda)}(x)\geq 
F_{U_n(\boldsymbol{s};\boldsymbol{\beta},\mu,\theta)}(x).
\end{eqnarray}
To establish the first two inequalities in (\ref{eq-3.25}), we have to show that  $F_{U_n(\boldsymbol{r};\boldsymbol{\alpha},\sigma,\lambda)}(x)$ and $F_{U_n(\boldsymbol{s};\boldsymbol{\alpha},\sigma,\lambda)}(x)$ are Schur-convex with respect to $\boldsymbol{r}$ and $\boldsymbol{\alpha}$, respectively. For $1\leq i<j\leq n$, we consider
\begin{align*}
\Delta_{1}=&\frac{\partial F_{U_n(\boldsymbol{r};\boldsymbol{\alpha},\sigma,\lambda)}(x)}{\partial r_i}-\frac{\partial F_{U_n(\boldsymbol{r};\boldsymbol{\alpha},\sigma,\lambda)}(x)}{\partial r_j}\\
=&F^{\alpha_i}\left(\frac{x-\sigma}{\lambda}\right)I(x>\sigma+c\lambda)-
F^{\alpha_j}\left(\frac{x-\sigma}{\lambda}\right)I(x>\sigma+c\lambda)\geq 0,
\end{align*}
whenever $\boldsymbol{\alpha}\in\mathcal{E}_n^+$. Thus, for $1\leq i<j\leq n$, we have $\Delta_{1}\geq 0$, implying that $\frac{\partial F_{U_n(\boldsymbol{r};\boldsymbol{\alpha},\sigma,\lambda)}(x)}{\partial r_k}$ is decreasing in $k\in\mathcal{I}_n$. From Lemma \ref{lemma2.3} with the help of Remark \ref{remark2.1}, it follows that $F_{U_n(\boldsymbol{r};\boldsymbol{\alpha},\sigma,\lambda)}(x)$ is Schur-convex with respect to $\boldsymbol{r}\in\mathcal{D}_n^+$. Thus, from Definition \ref{def2.4} 
\begin{eqnarray}\label{eq3.25}
\boldsymbol{r}\stackrel{m}{\succcurlyeq}\boldsymbol{s}\Rightarrow F_{U_n(\boldsymbol{r};\boldsymbol{\alpha},\sigma,\lambda)}(x)\geq F_{U_n(\boldsymbol{s};\boldsymbol{\alpha},\sigma,\lambda)}(x).
\end{eqnarray}
Furthermore, for $1\leq i<j\leq n$, we consider
\begin{eqnarray*}
\Delta_{2}&=&\frac{\partial F_{U_n(\boldsymbol{s};\boldsymbol{\alpha},\sigma,\lambda)}(x)}{\partial \alpha_i}-\frac{\partial F_{U_n(\boldsymbol{s};\boldsymbol{\alpha},\sigma,\lambda)}(x)}{\partial\alpha_j}\\
&=&s_iF^{\alpha_i}\left(\frac{x-\sigma}{\lambda}\right)\log\left(F\left( \frac{x-\sigma}{\lambda}\right)\right)I(x>\sigma+c\lambda)-s_jF^{\alpha_j}\left( \frac{x-\sigma}{\lambda}\right)\log\left(F\left(\frac{x-\sigma}{\lambda}\right)\right)\\
~~&\times&I(x>\sigma+c\lambda)\leq 0,
\end{eqnarray*}
whenever $\boldsymbol{s}\in\mathcal{D}_n^+$ and $\boldsymbol{\alpha}\in\mathcal{E}_n^+$. Thus, for $1\leq i<j\leq n$, we have $\Delta_{2}\leq 0$, implying that $\frac{\partial F_{U_n(\boldsymbol{s};\boldsymbol{\alpha},\sigma,\lambda)}(x)}{\partial\alpha_k}$ is increasing in $k\in\mathcal{I}_n$. Now, using Remark \ref{remark2.1} and Lemma \ref{lemma2.4}, we see that $F_{U_n(\boldsymbol{s};\boldsymbol{\alpha},\sigma,\lambda)}(x)$ is Schur-convex with respect to $\boldsymbol{\alpha}\in\mathcal{E}_n^+$. Thus, from Definition \ref{def2.4} 
\begin{eqnarray}\label{eq3.26}
\boldsymbol{\alpha}\stackrel{m}{\succcurlyeq}\boldsymbol{\beta}\Rightarrow F_{U_n(\boldsymbol{s};\boldsymbol{\alpha},\sigma,\lambda)}(x)\geq F_{U_n(\boldsymbol{s};\boldsymbol{\beta},\sigma,\lambda)}(x).
\end{eqnarray}
Furthermore, by the assumptions $\sigma\leq\mu$ and $\lambda\leq\theta$, we obtain $F^{\beta_i}(\frac{x-\sigma}{\lambda})I(x>\sigma+c\lambda)\geq F^{\beta_i}(\frac{x-\mu}{\theta})I(x>\mu+c\theta)$, for $i\in\mathcal{I}_n$, from which it is clear that
\begin{eqnarray}\label{eq3.27}
F_{U_n(\boldsymbol{s};\boldsymbol{\beta},\sigma,\lambda)}(x)\geq F_{U_n(\boldsymbol{s};\boldsymbol{\beta},\mu,\theta)}(x).
\end{eqnarray}
Now, upon combining the inequalities given in (\ref{eq3.25}), (\ref{eq3.26}), and (\ref{eq3.27}), the required usual stochastic ordering result readily follows. 
\end{proof}
	
\section{Stochastic comparisons between two multiple-outlier FMMs\setcounter{equation}{0}}\label{section4}
Here our focus is on stochastic comparisons of two FMMs with multiple-outlier ELS family distributed components. For this, we consider two FMMs with $n$ and $n^*$ numbers of RVs (components), respectively. In the first FMM, we assume that $n_1$ RVs are taken from a particular homogeneous subpopulation and the remaining $n_2$ RVs are taken from another type of homogeneous subpopulation, such that $n_1+n_2=n$. Similarly, in the second FMM, we consider that $n_1^*$ RVs are taken from a certain homogeneous subpopulation and the remaining $n_2^*$ RVs are drawn from another kind of homogeneous subpopulation, where $n_1^*+n_2^*=n^*$. The reversed hazard rate, likelihood ratio, and AFO in reversed hazard rate are investigated in this section. Before proceeding further, we introduce the following assumption.
\begin{assumption}\label{assumption4.1}
Consider $X_1,\ldots,X_{n_1}$ is a random sample of size $n_1$ taken from an absolutely continuous nonnegative RV $X^{(1)}\sim ELS(F,\alpha_1,\sigma_1,\lambda_1)$, and $X_{n_1+1},\ldots,X_n$ is another independent random sample of size $n_2$ from an absolutely continuous nonnegative RV $X^{(2)}\sim ELS(F,\alpha_2,\sigma_2,\lambda_2)$, where $n_1+n_2=n$. Further, assume that $Y_1,\ldots,Y_{n^*_1}$ is a random sample of size $n^*_1$ taken from an absolutely continuous nonnegative RV $Y^{(1)}\sim ELS(F,\beta_1,\mu_1,\theta_1)$, and $Y_{n^*_1+1},\ldots,Y_{n^*}$ is another independent random sample of size $n^*_2$ from an absolutely continuous nonnegative RV $Y^{(2)}\sim ELS(F,\beta_2,\mu_2,\theta_2)$, where $n^*_1+n^*_2=n^*$. Represent FM of $X_1,\ldots,X_n$ by $U$ with mixing proportions $r_i=r_1$ for $i=1,\ldots,n_1$ and $r_i=r_2$ for $i=n_1+1,\ldots,n$ such that $n_1r_1+n_2r_2=1$. Similarly, $V$ represents FM of $Y_1,\ldots,Y_n$ with mixing proportions $s_i=s_1$ for $i=1,\ldots,n^*_1$ and $s_i=s_2$ for $i=n^*_1+1,\ldots,n^*$ such that $n^*_1s_1+n^*_2s_2=1$.
\end{assumption}

In the next theorem, we derive the stochastic comparison result between $	U_{n}(\boldsymbol{r},\boldsymbol{\alpha},\boldsymbol{\sigma},\boldsymbol{\lambda})$ and $U_{n^*}(\boldsymbol{s},\boldsymbol{\alpha},\boldsymbol{\sigma},\boldsymbol{\lambda})$ with respect to the reversed hazard rate order. Throughout this section, bold symbols are used to denote the vectors having only two components. For example, $\boldsymbol{r}=(r_1,r_2)$, $\boldsymbol{s}=(s_1,s_2)$, $\boldsymbol{\alpha}=(\alpha_1,\alpha_2)$, $\boldsymbol{\sigma}=(\sigma_1,\sigma_2)$, and $\boldsymbol{\lambda}=(\lambda_1,\lambda_2)$. In this result, we have considered the number of the RVs in the multiple-outlier models are different, that is, $n_i\neq n_i^*$, for $i\in\mathcal{I}_2$. Here, the mixing proportion parameter vectors are different for both MMs.

\begin{theorem}\label{theorem4.1}
Under Assumption \ref{assumption4.1}, consider $X^{(i)}\stackrel{st}{=}Y^{(i)}$, for $i\in\mathcal{I}_2$. Assume that $\boldsymbol{\alpha}\in\mathcal{E}^+_2~(\mathcal{D}^+_2)$, $\boldsymbol{\lambda}\in\mathcal{E}^+_2~(\mathcal{D}^+_2)$, $\boldsymbol{\sigma}\in\mathcal{E}^+_2~(\mathcal{D}^+_2)$, and $t\tilde{h}(t)$ is decreasing in $t>c\geq 0$. Then, we have
\begin{eqnarray*}
U_{n}(\boldsymbol{r},\boldsymbol{\alpha},\boldsymbol{\sigma},\boldsymbol{\lambda})\leq_{rh}	U_{n^*}(\boldsymbol{s},\boldsymbol{\alpha},\boldsymbol{\sigma},\boldsymbol{\lambda}),	
\end{eqnarray*} 
provided that $n_1r_1{n_2^*}{s_2}\geq(\leq)~n_2r_2{n_1^*}{s_1}$.
\end{theorem}
	
\begin{proof}
For $\boldsymbol{\sigma}\in\mathcal{E}^+_2$ and $\boldsymbol{\lambda}\in\mathcal{E}_2^+$, the RHRF for the MRV $	U_{n}(\boldsymbol{r},\boldsymbol{\alpha},\boldsymbol{\sigma},\boldsymbol{\lambda})$ is obtained as
\begin{eqnarray*}
\tilde{h}_{	U_{n}(\boldsymbol{r},\boldsymbol{\alpha},\boldsymbol{\sigma},\boldsymbol{\lambda})}(x)=
\begin{cases}
0,&~if~x\leq\sigma_1+c\lambda_1;\\
\frac{\alpha_1f\left(\frac{x-\sigma_1}{\lambda_1}\right) }{\lambda_1F\left(\frac{x-\sigma_1}{\lambda_1}\right)}, &~if~\sigma_1+c\lambda_1<x\leq\sigma_2+c\lambda_2;\\
\frac{\frac{n_1r_1\alpha_1}{\lambda_1}F^{\alpha_1-1}\left(\frac{x-\sigma_1}{\lambda_1}\right)f\left(\frac{x-\sigma_1}{\lambda_1}\right)+\frac{n_2r_2\alpha_2}{\lambda_2}F^{\alpha_2-1}\left(\frac{x-\sigma_2}{\lambda_2}\right)f\left(\frac{x-\sigma_2}{\lambda_2}\right) }{n_1r_1F^{\alpha_1}\left(\frac{x-\sigma_1}{\lambda_1}\right)+n_2r_2F^{\alpha_2}\left(\frac{x-\sigma_2}{\lambda_2}\right)},&~if~x>\sigma_2+c\lambda_2.
\end{cases}
\end{eqnarray*}
To obtained the desired result, it is necessary to prove that $\tilde{h}_{	U_{n}(\boldsymbol{r},\boldsymbol{\alpha},\boldsymbol{\sigma},\boldsymbol{\lambda})}(x)\leq\tilde{h}_{	U_{n^*}(\boldsymbol{s},\boldsymbol{\alpha},\boldsymbol{\sigma},\boldsymbol{\lambda})}(x)$,
for all $x>\sigma_1+c\lambda_1$, where $c\in\mathbb{R^+}$. Firstly, in the subintervals $x\leq\sigma_1+c\lambda_1$ and $\sigma_1+c\lambda_1<x\leq\sigma_2+c\lambda_2$, it is evident that $\tilde{h}_{	U_{n}(\boldsymbol{r},\boldsymbol{\alpha},\boldsymbol{\sigma},\boldsymbol{\lambda})}(x)=\tilde{h}_{	U_{n^*}(\boldsymbol{s},\boldsymbol{\alpha},\boldsymbol{\sigma},\boldsymbol{\lambda})}(x).$ 
Secondly, for the subinterval $x>\sigma_2+c\lambda_2$, we are required to show that
\begin{align}\label{eq4.10}
&\frac{\frac{n_1r_1\alpha_1}{\lambda_1}F^{\alpha_1-1}\left(\frac{x-\sigma_1}{\lambda_1}\right)f\left(\frac{x-\sigma_1}{\lambda_1}\right)+\frac{n_2r_2\alpha_2}{\lambda_2}F^{\alpha_2-1}\left(\frac{x-\sigma_2}{\lambda_2}\right)f\left(\frac{x-\sigma_2}{\lambda_2}\right) }{n_1r_1F^{\alpha_1}\left(\frac{x-\sigma_1}{\lambda_1}\right)+n_2r_2F^{\alpha_2}\left(\frac{x-\sigma_2}{\lambda_2}\right)}\nonumber\\
\leq~&
\frac{\frac{n_1^*s_1\alpha_1}{\lambda_1}F^{\alpha_1-1}\left(\frac{x-\sigma_1}{\lambda_1}\right)f\left(\frac{x-\sigma_1}{\lambda_1}\right)+\frac{n_2^*s_2\alpha_2}{\lambda_2}F^{\alpha_2-1}\left(\frac{x-\sigma_2}{\lambda_2}\right)f\left(\frac{x-\sigma_2}{\lambda_2}\right) }{n_1^*s_1F^{\alpha_1}\left(\frac{x-\sigma_1}{\lambda_1}\right)+n_2^*s_2F^{\alpha_2}\left(\frac{x-\sigma_2}{\lambda_2}\right)}.
\end{align}
After simplifying the above inequality in (\ref{eq4.10}), we obtain
\begin{eqnarray*}
\left(n_1r_1n_2^*s_2-n_2r_2n_1^*s_1 \right)F^{\alpha_1}\left(\frac{x-\sigma_1}{\lambda_1} \right)
F^{\alpha_2}\left(\frac{x-\sigma_2}{\lambda_2} \right)\left[\frac{\alpha_1}{\lambda_1}\frac{f\left(\frac{x-\sigma_1}{\lambda_1} \right) }{F\left(\frac{x-\sigma_1}{\lambda_1} \right) }-\frac{\alpha_2}{\lambda_2}\frac{f\left(\frac{x-\sigma_2}{\lambda_2} \right) }{F\left(\frac{x-\sigma_2}{\lambda_2} \right) } \right] \leq 0,
\end{eqnarray*} 
which is equivalent to
\begin{eqnarray}\label{eq4.11}
\frac{1}{\lambda_1}\tilde{h}\left(\frac{x-\sigma_1}{\lambda_1}\right) \leq\frac{1}{\lambda_2}\tilde{h}\left(\frac{x-\sigma_2}{\lambda_2}\right), 
\end{eqnarray}
since $n_1r_1n_2^*s_2\geq n_2r_2n_1^*s_1$ and $\alpha_1\leq\alpha_2$ by the assumptions. Also, we can observe that
\begin{eqnarray}\label{eq4.15*}
\frac{1}{\lambda_i}\tilde{h}\left(\frac{x-\sigma_i}{\lambda_i} \right)=\frac{1}{x-\sigma_i}\left(\frac{x-\sigma_i}{\lambda_i} \right)\tilde{h}\left(\frac{x-\sigma_i}{\lambda_i} \right),  
\end{eqnarray}
for $i\in\mathcal{I}_2$. Now, under the assumptions, we have $\lambda_1\leq\lambda_2$, $\sigma_1\leq\sigma_2$, and $t\tilde{h}(t)$ is decreasing in $t>c$. Thus, after some simplification, using (\ref{eq4.15*}), the inequality in $(\ref{eq4.11})$ can be easily established. Hence, it follows that the inequality in $(\ref{eq4.10})$ also holds, which completes the proof. 
\end{proof}
	
The preceding theorem has established the reversed hazard rate ordering between two multiple-outlier MRVs. Thus, it is natural to examine if the reversed hazard rate order in Theorem \ref{theorem4.1} can be extended to some other stronger stochastic ordering results. In the following theorem, we have established the likelihood ratio order between $U_{n}(\boldsymbol{r},\boldsymbol{\alpha},\boldsymbol{\sigma},\boldsymbol{\lambda})$ and $	U_{n^*}(\boldsymbol{s},\boldsymbol{\alpha},\boldsymbol{\sigma},\boldsymbol{\lambda})$.
 
\begin{theorem}\label{theorem4.2}
Under Assumption \ref{assumption4.1}, consider $X^{(i)}\stackrel{st}{=}Y^{(i)}$, for $i\in\mathcal{I}_2$. Assume that $\boldsymbol{\alpha},\boldsymbol{\lambda},\boldsymbol{\sigma}\in\mathcal{E}^+_2$, $t\tilde{h}(t)$, and $\frac{tf^{\prime}(t)}{f(t)}$ are decreasing in $t>c\geq 0$. Then, we have
\begin{eqnarray*}
U_{n}(\boldsymbol{r},\boldsymbol{\alpha},\boldsymbol{\sigma},\boldsymbol{\lambda})\geq_{lr}	U_{n^*}(\boldsymbol{s},\boldsymbol{\alpha},\boldsymbol{\sigma},\boldsymbol{\lambda}),	
\end{eqnarray*}
provided that $\alpha_1\geq 1$, $\alpha_2\geq 1$, and $n_1r_1n_2^*s_2\leq n_2r_2n_1^*s_1$.
\end{theorem}

\begin{proof}
For $\boldsymbol{\sigma},\boldsymbol{\lambda}\in\mathcal{E}_2^+$, the PDF of the MRV $U_{n}(\boldsymbol{r},\boldsymbol{\alpha},\boldsymbol{\sigma},\boldsymbol{\lambda})$ is written as 
\begin{eqnarray*}
f_{	U_{n}(\boldsymbol{r},\boldsymbol{\alpha},\boldsymbol{\sigma},\boldsymbol{\lambda})}(x)=
\begin{cases}
0,&if~x\leq\sigma_1+c\lambda_1;\\
\frac{n_1r_1\alpha_1}{\lambda_1}F^{\alpha_1-1}\left(\frac{x-\sigma_1}{\lambda_1}\right)f\left(\frac{x-\sigma_1}{\lambda_1}\right),&if~\sigma_1+c\lambda_1<x\leq\sigma_2+c\lambda_2;\\
\sum_{i=1}^{2}\frac{n_ir_i\alpha_i}{\lambda_i}F^{\alpha_i-1}\left(\frac{x-\sigma_i}{\lambda_i}\right)f\left(\frac{x-\sigma_i}{\lambda_i}\right),&if~x>\sigma_2+c\lambda_2.
\end{cases}
\end{eqnarray*}
To get the desired result, we have to demonstrate that $\zeta(x)=f_{	U_{n}(\boldsymbol{r},\boldsymbol{\alpha},\boldsymbol{\sigma},\boldsymbol{\lambda})}(x)/f_{	U_{n^*}(\boldsymbol{s},\boldsymbol{\alpha},\boldsymbol{\sigma},\boldsymbol{\lambda})}(x)$ is increasing in $x>\sigma_1+c\lambda_1$, where $c\in\mathbb{R}^+$. Firstly, consider the subinterval $\sigma_1+c\lambda_1<x\leq\sigma_2+c\lambda_2$. Now, in this subinterval, we observe that 
\begin{eqnarray}\label{eq4.13}
\zeta(x)=\frac{n_1r_1}{n_1^*s_1},
\end{eqnarray}
which is a constant (independent of $x$), thus can be taken as an increasing function in wide-sense. Secondly, for the subinterval $x>\sigma_2+c\lambda_2$, we have to show that
\begin{eqnarray}\label{eq4.14}
\zeta(x)=\frac{\frac{n_1r_1\alpha_1}{\lambda_1}F^{\alpha_1-1}\left(\frac{x-\sigma_1}{\lambda_1}\right)f\left(\frac{x-\sigma_1}{\lambda_1}\right)+\frac{n_2r_2\alpha_2}{\lambda_2}F^{\alpha_2-1}\left(\frac{x-\sigma_2}{\lambda_2}\right)f\left(\frac{x-\sigma_2}{\lambda_2}\right)}{\frac{n_1^*r_1^*\alpha_1}{\lambda_1}F^{\alpha_1-1}\left(\frac{x-\sigma_1}{\lambda_1}\right)f\left(\frac{x-\sigma_1}{\lambda_1}\right)+\frac{n_2^*r_2^*\alpha_2}{\lambda_2}F^{\alpha_2-1}\left(\frac{x-\sigma_2}{\lambda_2}\right)f\left(\frac{x-\sigma_2}{\lambda_2}\right)}
\end{eqnarray}
is increasing in $x$. Note that the ratio given in (\ref{eq4.14}) can be written as

\begin{eqnarray}\label{eq4.15}
\zeta(x)=\frac{n_1r_1\phi(x)+n_2r_2}{n_1^*s_1\phi(x)+n_2^*s_2}=\frac{1}{n_1^*s_1}\left[n_1r_1+\frac{n_2r_2n_1^*s_1-n_1r_1n_2^*s_2}{n_1^*s_1\phi(x)+n_2^*s_2}\right],
\end{eqnarray}
where $\phi(x)=\frac{\frac{\alpha_1}{\lambda_1}F^{\alpha_1-1}\left(\frac{x-\sigma_1}{\lambda_1}\right)f\left(\frac{x-\sigma_1}{\lambda_1}\right) }{\frac{\alpha_2}{\lambda_2}F^{\alpha_2-1}\left(\frac{x-\sigma_2}{\lambda_2}\right) f\left(\frac{x-\sigma_2}{\lambda_2}\right)}$.
Differentiating (\ref{eq4.15}) partially with respect to $x$, we obtain
\begin{eqnarray}\label{eq4.17}
\zeta^{\prime}(x)=\frac{(n_1r_1n_2^*s_2-n_2r_2n_1^*s_1)\phi^{\prime}(x)}{(n_1^*s_1\phi(x)+n_2^*s_2)^2}.
\end{eqnarray}
Therefore, (\ref{eq4.15}) is increasing in $x$, if $\phi(x)$ is decreasing in $x$, since $n_1r_1n_2^*s_2\leq n_2r_2n_1^*s_1$ by assumption. Now, differentiating $\phi(x)$ with respect to $x$, we obtain
\allowdisplaybreaks{\begin{eqnarray}\label{eq4.18}
\phi^{\prime}(x)&=&\frac{\frac{\alpha_1}{\lambda_1}}{\frac{\alpha_2}{\lambda_2}\left[F^{\alpha_2-1}\left(\frac{x-\sigma_2}{\lambda_2}\right)f\left(\frac{x-\sigma_2}{\lambda_2}\right)\right]^2}
\Bigg[ \frac{1}{\lambda_1}F^{\alpha_2-1}\left(\frac{x-\sigma_2}{\lambda_2}\right)f\left(\frac{x-\sigma_2}{\lambda_2}\right)
\Bigg\{ (\alpha_1-1)F^{\alpha_1-2}\left(\frac{x-\sigma_1}{\lambda_1}\right)\nonumber\\
&&
\times f\left(\frac{x-\sigma_1}{\lambda_1}\right)^2+F^{\alpha_1-1}\left(\frac{x-\sigma_1}{\lambda_1}\right)f^{\prime}\left(\frac{x-\sigma_1}{\lambda_1}\right)\Bigg\}-
\frac{1}{\lambda_2}F^{\alpha_1-1}\left(\frac{x-\sigma_1}{\lambda_1}\right)f\left(\frac{x-\sigma_1}{\lambda_1}\right)\nonumber\\
&&
\times \Bigg\{ (\alpha_2-1)F^{\alpha_2-2}\left(\frac{x-\sigma_2}{\lambda_2}\right)
f\left(\frac{x-\sigma_2}{\lambda_2}\right)^2+F^{\alpha_2-1}\left(\frac{x-\sigma_2}{\lambda_2}\right)f^{\prime}\left(\frac{x-\sigma_2}{\lambda_2}\right)\Bigg\} \Bigg]\nonumber\\ 
&=&\frac{\frac{\alpha_1}{\lambda_1}}{\frac{\alpha_2}{\lambda_2}}\frac{F^{\alpha_1-1}\left(\frac{x-\sigma_1}{\lambda_1}\right)f\left(\frac{x-\sigma_1}{\lambda_1}\right) F^{\alpha_2-1}\left(\frac{x-\sigma_2}{\lambda_2}\right)f\left(\frac{x-\sigma_2}{\lambda_2}\right)}{\left[F^{\alpha_2-1}\left(\frac{x-\sigma_2}{\lambda_2}\right)f\left(\frac{x-\sigma_2}{\lambda_2}\right)\right]^2}\nonumber\\
&&\times\Bigg[\frac{\alpha_1-1}{\lambda_1}\frac{f\left( \frac{x-\sigma_1}{\lambda_1}\right)}{F\left( \frac{x-\sigma_1}{\lambda_1}\right)}-\frac{\alpha_2-1}{\lambda_2}\frac{f\left( \frac{x-\sigma_2}{\lambda_2}\right)}{F\left( \frac{x-\sigma_2}{\lambda_2}\right)}
+\frac{1}{\lambda_1}\frac{f^{\prime}\left(\frac{x-\sigma_1}{\lambda_1} \right) }{f\left(\frac{x-\sigma_1}{\lambda_1} \right)}-\frac{1}{\lambda_2}\frac{f^{\prime}\left(\frac{x-\sigma_2}{\lambda_2} \right) }{f\left(\frac{x-\sigma_2}{\lambda_2} \right)}\Bigg]\nonumber\\
&\stackrel{sign}{=}&\frac{\alpha_1-1}{x-\sigma_1}\frac{x-\sigma_1}{\lambda_1}\frac{f\left( \frac{x-\sigma_1}{\lambda_1}\right)}{F\left( \frac{x-\sigma_1}{\lambda_1}\right)}-\frac{\alpha_2-1}{x-\sigma_2}\frac{x-\sigma_2}{\lambda_2}\frac{f\left( \frac{x-\sigma_2}{\lambda_2}\right)}{F\left( \frac{x-\sigma_2}{\lambda_2}\right)}\nonumber\\
&&+\frac{1}{x-\sigma_1}\frac{x-\sigma_1}{\lambda_1}\frac{f^{\prime}\left(\frac{x-\sigma_1}{\lambda_1} \right) }{f\left(\frac{x-\sigma_1}{\lambda_1} \right)}
-\frac{1}{x-\sigma_2}\frac{x-\sigma_2}{\lambda_2}\frac{f^{\prime}\left(\frac{x-\sigma_2}{\lambda_2} \right) }{f\left(\frac{x-\sigma_2}{\lambda_2} \right)}.
\end{eqnarray}}
Under the assumptions, $1\leq\alpha_1\leq\alpha_2$, $\sigma_1\leq\sigma_2$, $\lambda_1\leq\lambda_2$, $t\tilde{h}(t)$, and $\frac{tf^{\prime}(t)}{f(t)}$ are decreasing in $t>c$, it can be shown that $\phi^{\prime}(x)\leq 0$. Hence, the proof is completed.
\end{proof}

In the next result, we establish AFO in terms of the reversed hazard rate between two multiple outlier MRVs. 
	
\begin{theorem}\label{theorem4.3}
Under Assumption \ref{assumption4.1}, let $\sigma_i=\sigma$, $\mu_i=\mu$, $\alpha_i=\alpha\in(0,1]$, and $c=0$, for $i\in\mathcal{I}_2$. Assume that $\sigma\geq\mu$ and $\min\{\lambda_1,\lambda_2\}\geq\max\{\theta_1,\theta_2\}$. Then, we have
\begin{eqnarray*}
U_{n}(\boldsymbol{r};{\alpha},\sigma,\boldsymbol{\lambda})\leq_{R-rh}V_{ n^*}(\boldsymbol{s};{\alpha},\mu,\boldsymbol{\theta}),
\end{eqnarray*}
provided that $\frac{f^\prime(t)}{f(t)}$ is increasing and $t\tilde{h}(t)$ is decreasing in $t>0$. 
\end{theorem}
	
\begin{proof}
For $\sigma_1=\sigma_2=\sigma$ and $c=0$, the reversed hazard rate for the MRV $U_{n}(\boldsymbol{r};{\alpha},\sigma,\boldsymbol{\lambda})$ is as follows
\begin{eqnarray}\label{eq4.19}
\tilde{h}_{U_{n}(\boldsymbol{r};{\alpha},\sigma,\boldsymbol{\lambda})}(x)=\frac{f_{U_{n}(\boldsymbol{r};{\alpha},\sigma,\boldsymbol{\lambda})}(x)}{F_{U_{n}(\boldsymbol{r};{\alpha},\sigma,\boldsymbol{\lambda})}(x)}=\frac{\sum_{i=1}^{2}\frac{n_ir_i\alpha_i}{\lambda_i}F^{\alpha-1}\left(\frac{x-\sigma}{\lambda_i}\right)f\left(\frac{x-\sigma}{\lambda_i} \right)}{\sum_{i=1}^{2}n_ir_iF^{\alpha}\left(\frac{x-\sigma}{\lambda_i} \right)}~\mbox{if}~x>\sigma.
\end{eqnarray}	
To get the necessary result, we need to show that $\Lambda(x)=\frac{\tilde{h}_{U_{n}(\boldsymbol{r};{\alpha},\sigma,\boldsymbol{\lambda})}(x)}{\tilde{h}_{V_{n^*}(\boldsymbol{r^*};{\alpha},\mu,\boldsymbol{\theta})}(x)}$ is decreasing in $x>\max\{\sigma,\mu\}$. For this purpose, we obtain 
		
\begin{eqnarray}
\Lambda^\prime(x)&=&\left( \frac{f_{U_{n}(\boldsymbol{r};{\alpha},\sigma,\boldsymbol{\lambda})}(x)F_{V_{ n^*}(\boldsymbol{s};{\alpha},\mu,\boldsymbol{\theta})}(x)}{F_{U_{n}(\boldsymbol{r};{\alpha},\sigma,\boldsymbol{\lambda})}(x)f_{V_{ n^*}(\boldsymbol{s};{\alpha},\mu,\boldsymbol{\theta})}(x)}\right)^\prime\nonumber\\
&=&\frac{1}{(F_{U_{n}(\boldsymbol{r};{\alpha},\sigma,\boldsymbol{\lambda})}(x)f_{V_{n^*}(\boldsymbol{s};{\alpha},\mu,\boldsymbol{\theta})}(x))^2} \Big[F_{U_{n}(\boldsymbol{r};{\alpha},\sigma,\boldsymbol{\lambda})}(x)f_{V_{ n^*}(\boldsymbol{s};{\alpha},\mu,\boldsymbol{\theta})}(x)\nonumber\\
&&\times\Big\{f^{\prime}_{U_{n}(\boldsymbol{r};{\alpha},\sigma,\boldsymbol{\lambda})}(x)F_{V_{ n^*}(\boldsymbol{s};{\alpha},\mu,\boldsymbol{\theta})}(x)+f_{U_{n}(\boldsymbol{r};{\alpha},\sigma,\boldsymbol{\lambda})}(x)f_{V_{ n^*}(\boldsymbol{s};{\alpha},\mu,\boldsymbol{\theta})}(x)\Big\}\nonumber\\
&&-f_{U_{n}(\boldsymbol{r};{\alpha},\sigma,\boldsymbol{\lambda})}(x)F_{V_{ n^*}(\boldsymbol{s};{\alpha},\mu,\boldsymbol{\theta})}(x)\Big\{f_{U_{n}(\boldsymbol{r};{\alpha},\sigma,\boldsymbol{\lambda})}(x)f_{V_{ n^*}(\boldsymbol{s};{\alpha},\mu,\boldsymbol{\theta})}(x)\nonumber\\
&&+F_{U_{n}(\boldsymbol{r};{\alpha},\sigma,\boldsymbol{\lambda})}(x)f^{\prime}_{V_{ n^*}(\boldsymbol{s};{\alpha},\mu,\boldsymbol{\theta})}(x)\Big\}\Big]\nonumber\\
&\stackrel{sign}{=}&F_{U_{n}(\boldsymbol{r};{\alpha},\sigma,\boldsymbol{\lambda})}(x)F_{V_{ n^*}(\boldsymbol{s};{\alpha},\mu,\boldsymbol{\theta})}(x)[f^\prime_{U_{n}(\boldsymbol{r};{\alpha},\sigma,\boldsymbol{\lambda})}(x)f_{V_{ n^*}(\boldsymbol{s};{\alpha},\mu,\boldsymbol{\theta})}(x)\nonumber\\
&&-f_{U_{n}(\boldsymbol{r};{\alpha},\sigma,\boldsymbol{\lambda})}(x)f^\prime_{V_{ n^*}(\boldsymbol{s};{\alpha},\mu,\boldsymbol{\theta})}(x)]+f_{U_{n}(\boldsymbol{r};{\alpha},\sigma,\boldsymbol{\lambda})}(x)f_{V_{ n^*}(\boldsymbol{s};{\alpha},\mu,\boldsymbol{\theta})}(x)\nonumber\\
&&\times[F_{U_{n}(\boldsymbol{r};{\alpha},\sigma,\boldsymbol{\lambda})}(x)f_{V_{ n^*}(\boldsymbol{s};{\alpha},\mu,\boldsymbol{\theta})}(x)-f_{U_{n}(\boldsymbol{r};{\alpha},\sigma,\boldsymbol{\lambda})}(x)F_{V_{ n^*}(\boldsymbol{s};{\alpha},\mu,\boldsymbol{\theta})}(x)].\nonumber
\end{eqnarray}
Thus, the ratio $\Lambda(x)$ is decreasing in $x>\max\{\sigma,\mu\}$, if
\begin{eqnarray}\label{eq4.20}
f^\prime_{U_{n}(\boldsymbol{r};{\alpha},\sigma,\boldsymbol{\lambda})}(x)f_{V_{ n^*}(\boldsymbol{s};{\alpha},\mu,\boldsymbol{\theta})}(x)\leq f_{U_{n}(\boldsymbol{r};{\alpha},\sigma,\boldsymbol{\lambda})}(x)f^\prime_{V_{ n^*}(\boldsymbol{s};{\alpha},\mu,\boldsymbol{\theta})}(x)
\end{eqnarray}
and
\begin{eqnarray}\label{eq4.21}
F_{U_{n}(\boldsymbol{r};{\alpha},\sigma,\boldsymbol{\lambda})}(x)f_{V_{ n^*}(\boldsymbol{s};{\alpha},\mu,\boldsymbol{\theta})}(x)\leq f_{U_{n}(\boldsymbol{r};{\alpha},\sigma,\boldsymbol{\lambda})}(x)F_{V_{ n^*}(\boldsymbol{s};{\alpha},\mu,\boldsymbol{\theta})}(x).
\end{eqnarray}
Now, the inequality in (\ref{eq4.20}) can be rewritten as
\begin{align}\label{eq4.22}
&\sum_{i=1}^{2}\sum_{j=1}^{2}\frac{\alpha^2n_ir_in_j^*s_j}{\lambda_i^2\theta_j}F^{\alpha-2}\left(\frac{x-\sigma}{\lambda_i}\right)F^{\alpha-1}\left(\frac{x-\mu}{\theta_j}\right)f\left(\frac{x-\mu}{\theta_j}\right) \Big[ (\alpha-1)f^2\left(\frac{x-\sigma}{\lambda_i}\right)\nonumber\\
&+F\left(\frac{x-\sigma}{\lambda_i}\right)f^\prime\left(\frac{x-\sigma}{\lambda_i}\right)\Big]\leq \sum_{i=1}^{2}\sum_{j=1}^{2}\frac{\alpha^2n_ir_in_j^*s_j}{\lambda_i\theta_j^2}F^{\alpha-1}\left(\frac{x-\sigma}{\lambda_i}\right)F^{\alpha-2}\left(\frac{x-\mu}{\theta_j}\right)f\left(\frac{x-\sigma}{\lambda_i}\right)\nonumber\\
& \times\Big[ (\alpha-1)f^2\left(\frac{x-\mu}{\theta_j}\right)+F\left(\frac{x-\mu}{\theta_j}\right)f^\prime\left(\frac{x-\mu}{\theta_j}\right)\Big]\nonumber\\
\Rightarrow&
\sum_{i=1}^{2}\sum_{j=1}^{2}\frac{\alpha^2n_ir_in_j^*s_j}{\lambda_i\theta_j}F^{\alpha-1}\left(\frac{x-\sigma}{\lambda_i}\right)F^{\alpha-1}\left(\frac{x-\mu}{\theta_j}\right)f\left( \frac{x-\sigma}{\lambda_i}\right)f\left(\frac{x-\mu}{\theta_j}\right) \Big[\frac{\alpha-1}{\lambda_i}\frac{f(\frac{x-\sigma}{\lambda_i})}{F(\frac{x-\sigma}{\lambda_i})}\nonumber\\
&-\frac{\alpha-1}{\theta_j}\frac{f(\frac{x-\mu}{\theta_j})}{F(\frac{x-\mu}{\theta_j})}+\frac{1}{\lambda_i}\frac{f^\prime(\frac{x-\sigma}{\lambda_i})}{f(\frac{x-\sigma}{\lambda_i})}-
\frac{1}{\theta_j}\frac{f^\prime(\frac{x-\mu}{\theta_j})}{f(\frac{x-\mu}{\theta_j})}\Big]\leq 0.
\end{align}
Also, the inequality in (\ref{eq4.21}) can be rewritten as
\begin{align}\label{eq4.23}
&\sum_{i=1}^{2}\sum_{j=1}^{2}\frac{n_ir_in_j^*s_j\alpha}{\theta_j}F^{\alpha}\left(\frac{x-\sigma}{\lambda_i}\right)F^{\alpha-1}\left(\frac{x-\mu}{\theta_j}\right)f\left(\frac{x-\mu}{\theta_j}\right)
\leq
\sum_{i=1}^{2}\sum_{j=1}^{2}\frac{n_ir_in_j^*s_j\alpha}{\lambda_i}F^{\alpha-1}\left(\frac{x-\sigma}{\lambda_i}\right)\nonumber\\
&\times F^{\alpha}\left(\frac{x-\mu}{\theta_j}\right)f\left(\frac{x-\sigma}{\lambda_i}\right)\nonumber\\
\Rightarrow&\sum_{i=1}^{2}\sum_{j=1}^{2}n_ir_in_j^*s_jF^{\alpha}\left(\frac{x-\sigma}{\lambda_i}\right)F^{\alpha}\left(\frac{x-\mu}{\theta_j}\right)\Bigg[\frac{\alpha}{\theta_j}\frac{f(\frac{x-\mu}{\theta_j})}{F(\frac{x-\mu}{\theta_j})}-\frac{\alpha}{\lambda_i}\frac{f(\frac{x-\sigma}{\lambda_i})}{F(\frac{x-\sigma}{\lambda_i})}\Bigg]\leq 0.
\end{align}
Now, using the assumptions $\sigma\geq\mu$, $\min\{\lambda_1,\lambda_2\}\geq\max\{\theta_1,\theta_2\}$, $\alpha\in(0,1]$, $\frac{f^{\prime}(t)}{f(t)}$ is increasing and $t\tilde{h}(t)$ is decreasing in $t>0$, we get the following inequalities
\begin{itemize}
\item[$(a)$] $\frac{\alpha-1}{\lambda_i}\frac{f(\frac{x-\sigma}{\lambda_i})}{F(\frac{x-\sigma}{\lambda_i})}\leq\frac{\alpha-1}{\theta_j}\frac{f(\frac{x-\mu}{\theta_j})}{F(\frac{x-\mu}{\theta_j})}$;
\item[$(b)$]
$\frac{1}{\lambda_i}\frac{f^\prime(\frac{x-\sigma}{\lambda_i})}{f(\frac{x-\sigma}{\lambda_i})}\leq\frac{1}{\theta_j}\frac{f^\prime(\frac{x-\mu}{\theta_j})}{f(\frac{x-\mu}{\theta_j})}$;
\item[$(c)$] $\frac{\alpha}{\lambda_i}\frac{f(\frac{x-\sigma}{\lambda_i})}{F(\frac{x-\sigma}{\lambda_i})}\geq\frac{\alpha}{\theta_j}\frac{f(\frac{x-\mu}{\theta_j})}{F(\frac{x-\mu}{\theta_j})}$,
\end{itemize}
for $i,j\in\mathcal{I}_2$. Note that, from $(a)$ and $(b)$, the inequality in (\ref{eq4.22}) follows. Further, from $(c)$, the inequality in (\ref{eq4.23}) can be established. Hence, the proof is completed.
\end{proof}
	
\section{Numerical examples and counterexamples}\label{section5}
Various examples and counterexamples are provided in this section to examine the main results established in Sections \ref{section3} and \ref{section4}. Theorem \ref{theorem3.1} is demonstrated by the following example. 
\begin{example}\label{example4.1}
Consider the Pareto distribution as the baseline distribution with CDF $F(x)=1-x^{-5},~x\geq1$. Let $\boldsymbol{r}=\left(0.40,0.55,0.05\right)$, 
$\boldsymbol{\lambda}=\left(2,4,6\right)\in\mathcal{E}^+_3$, $\boldsymbol{\theta}=\left(6,7,9\right)\in\mathcal{E}^+_3$, 
$\boldsymbol{\sigma}=\left(1,2,3\right)\in\mathcal{E}^+_3$, 
$\boldsymbol{\mu}=\left(2,4,5\right)\in\mathcal{E}^+_3$, 
$\boldsymbol{\alpha}=\left(5,2,7\right)$, and 
$\boldsymbol{\beta}=\left(9,10,8\right)$. Here, it is clear that $\lambda_i\leq\theta_i,~\sigma_i\leq\mu_i,$ and $\alpha_i\leq\beta_i$, for $i\in\mathcal{I}_3$. In Figure \ref{figure1}, we have plotted the graphs of the SFs of $U_3(\boldsymbol{r};\boldsymbol{\alpha},\boldsymbol{\sigma},\boldsymbol{\lambda})$ and $U_3(\boldsymbol{r};\boldsymbol{\beta},\boldsymbol{\mu},\boldsymbol{\theta})$. From this figure, it is observed that $U_3(\boldsymbol{r};\boldsymbol{\alpha},\boldsymbol{\sigma},\boldsymbol{\lambda})\leq_{st}U_3(\boldsymbol{r};\boldsymbol{\beta},\boldsymbol{\mu},\boldsymbol{\theta})$, validating the established result in Theorem \ref{theorem3.1}.  
\end{example}
	
Next, we consider a counterexample to show that the conditions ``$\lambda_i\leq\theta_i,~\sigma_i\leq\mu_i,~\mbox{and}~\alpha_i\leq\beta_i,~\mbox{for}~i\in\mathcal{I}_3$'' are needed to get the required usual stochastic ordering in Theorem \ref{theorem3.1}.	
	
\begin{counterexample}\label{Coun4.1}
Consider the baseline CDF as $F(x)=1-(\frac{4}{x})^{2},~x\geq4$. Further, let $\boldsymbol{r}=\left(0.1,0.3,0.6\right)$,  
$\boldsymbol{\sigma}=\left(2,6,7\right)\in\mathcal{E}_3^+$, $\boldsymbol{\lambda}=\left(1,5,8\right)\in\mathcal{E}_3^+$, 
$\boldsymbol{\mu}=\left(4,5,6\right)\in\mathcal{E}_3^+$, 
$\boldsymbol{\theta}=\left(2,3,4\right)\in\mathcal{E}_3^+$, 
$\boldsymbol{\alpha}=\left(7,6,1\right)$, and 
$\boldsymbol{\beta}=\left(6,9,9\right)$. Here, $\alpha_i\nleq\beta_i$, $\sigma_i\nleq\mu_i$, and $\lambda_i\nleq\theta_i$, for some $i\in\mathcal{I}_3$. In Figure \ref{figure2}, the graphs of the SFs of $U_3(\boldsymbol{r};\boldsymbol{\alpha},\boldsymbol{\sigma},\boldsymbol{\lambda})$ and $U_3(\boldsymbol{r};\boldsymbol{\beta},\boldsymbol{\mu},\boldsymbol{\theta})$ are presented, from which clearly $U_3(\boldsymbol{r};\boldsymbol{\alpha},\boldsymbol{\sigma},\boldsymbol{\lambda})\nleq_{st}U_3(\boldsymbol{r};\boldsymbol{\beta},\boldsymbol{\mu},\boldsymbol{\theta})$, not validating the established result in Theorem \ref{theorem3.1}.  
\end{counterexample}
	
\begin{figure}[h!]
\begin{center}
\subfigure[]{\label{figure1}\includegraphics[width=3.2in]{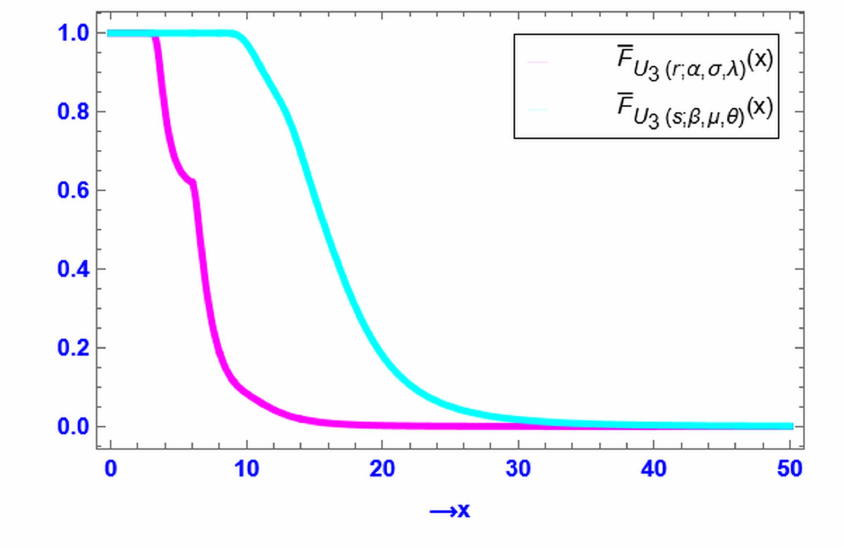}}
\subfigure[]{\label{figure2}\includegraphics[width=3.2in]{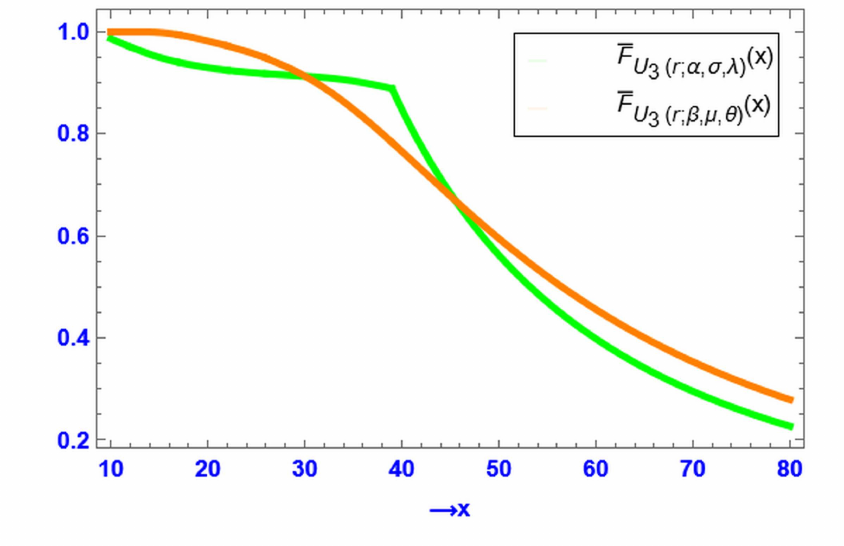}}
\caption{$(a)$ Graph of the SFs of $U_3(\boldsymbol{r};\boldsymbol{\alpha},\boldsymbol{\sigma},\boldsymbol{\lambda})$ (magenta curve) and $U_3(\boldsymbol{r};\boldsymbol{\beta},\boldsymbol{\mu},\boldsymbol{\theta})$ (cyan curve) in Example \ref{example4.1}. 
$(b)$ Representation of the SFs of $U_3(\boldsymbol{r};\boldsymbol{\alpha},\boldsymbol{\sigma},\boldsymbol{\lambda})$ (green curve) and $U_3(\boldsymbol{r};\boldsymbol{\beta},\boldsymbol{\mu},\boldsymbol{\theta})$ (orange curve) in Counterexample \ref{Coun4.1}.}
\end{center}
\end{figure}
	
Next, the following two consecutive counterexamples show that the result in Theorem \ref{theorem3.1} can not be extended from the usual stochastic ordering to the reversed hazard rate and likelihood ratio orderings under the same setup as considered in Theorem \ref{theorem3.1}.
	
\begin{counterexample}\label{Coun4.2}
Consider the Pareto distribution as the baseline distribution with CDF $F(x)=1-(\frac{2}{x})^{6},~x\geq2$. Set $\boldsymbol{r}=(0.2,0.3,0.5)$, 
$\boldsymbol{\lambda}=(1,4,6)\in\mathcal{E}^+_3$, $\boldsymbol{\theta}=(2,8,10)\in\mathcal{E}^+_3$, 
$\boldsymbol{\sigma}=(3,10,14)\in\mathcal{E}^+_3$, 
$\boldsymbol{\mu}=(6,11,16)\in\mathcal{E}^+_3$, 
$\boldsymbol{\alpha}=(4,6,5)$, and 
$\boldsymbol{\beta}=(8,7,11)$. Here, $\lambda_i\leq\theta_i$, $\sigma_i\leq\mu_i$, and  $\alpha_i\leq\beta_i$ for $i\in\mathcal{I}_3$. Figure \ref{figure3} displays the plot of the ratio of the CDFs of $U_3(\boldsymbol{r};\boldsymbol{\beta},\boldsymbol{\mu},\boldsymbol{\theta})$ and $U_3(\boldsymbol{r};\boldsymbol{\alpha},\boldsymbol{\sigma},\boldsymbol{\lambda})$, which evidently exhibits a non-monotonic behavior. This indicates that, under the same conditions as stated in Theorem \ref{theorem3.1}, the reversed hazard rate ordering between the MRVs $U_3(\boldsymbol{r};\boldsymbol{\alpha},\boldsymbol{\sigma},\boldsymbol{\lambda})$ and $U_3(\boldsymbol{r};\boldsymbol{\beta},\boldsymbol{\mu},\boldsymbol{\theta})$ fails to be satisfied. 
\end{counterexample}
	
\begin{counterexample}\label{Coun4.22}
Consider the baseline distribution with PDF $f(x)=3.5^3x^{-4},~x\geq5$. Assume that $\boldsymbol{r}=(0.3,0.2,0.5)$, 
$\boldsymbol{\lambda}=(1,2,5)\in\mathcal{E}^+_3$, $\boldsymbol{\theta}=(2,4,6)\in\mathcal{E}^+_3$, 
$\boldsymbol{\sigma}=(1,2,7)\in\mathcal{E}^+_3$, 
$\boldsymbol{\mu}=(5,7,8)\in\mathcal{E}^+_3$, 
$\boldsymbol{\alpha}=(2,6,5)$, and 
$\boldsymbol{\beta}=(6,7,9)$. Clearly, $\lambda_i\leq\theta_i$, $\sigma_i\leq\mu_i$, and $\alpha_i\leq\beta_i$, for $i\in\mathcal{I}_3$. In Figure \ref{figure4}, the graph of the ratio of the PDFs of $U_3(\boldsymbol{r};\boldsymbol{\beta},\boldsymbol{\mu},\boldsymbol{\theta})$ and $U_3(\boldsymbol{r};\boldsymbol{\alpha},\boldsymbol{\sigma},\boldsymbol{\lambda})$ is depicted. In this graph, we notice that the ratio is non-monotone, revealing the non-existence of likelihood ratio order between $U_3(\boldsymbol{r};\boldsymbol{\alpha},\boldsymbol{\sigma},\boldsymbol{\lambda})$ and $U_3(\boldsymbol{r};\boldsymbol{\beta},\boldsymbol{\mu},\boldsymbol{\theta})$, under the same setup in Theorem \ref{theorem3.1}.
\end{counterexample}
	
\begin{figure}[h!]
\begin{center}
\subfigure[]{\label{figure3}\includegraphics[width=3.2in]{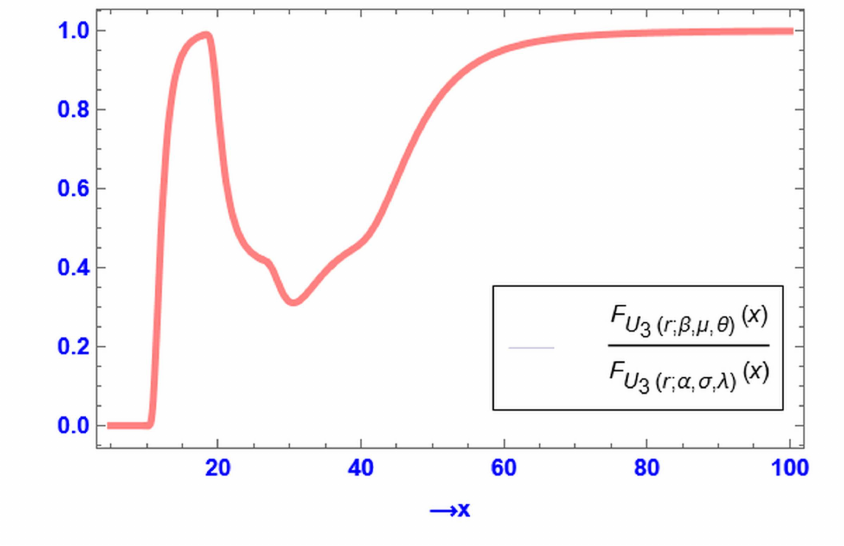}}
\subfigure[]{\label{figure4}\includegraphics[width=3.2in]{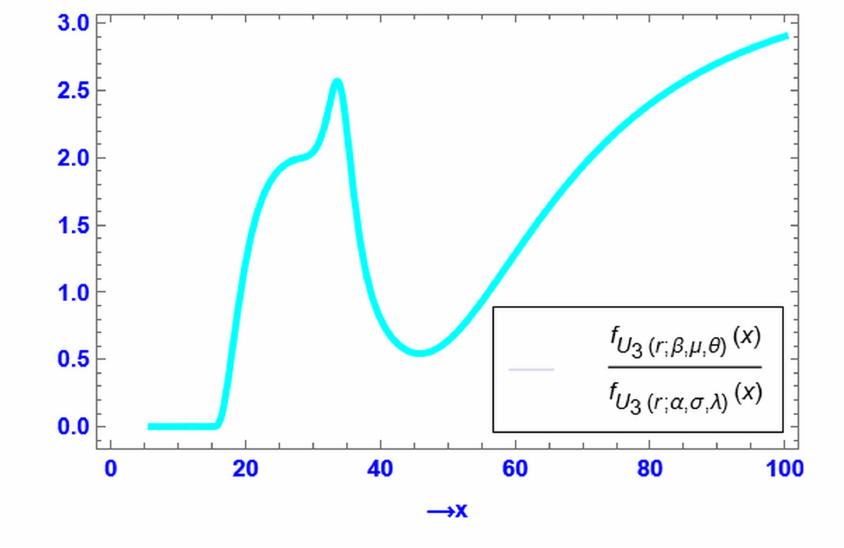}}
\caption{$(a)$ Graph of the ratio between the CDFs of $U_3(\boldsymbol{r};\boldsymbol{\alpha},\boldsymbol{\sigma},\boldsymbol{\lambda})$ and $U_3(\boldsymbol{r};\boldsymbol{\beta},\boldsymbol{\mu},\boldsymbol{\theta})$ in Counterexample \ref{Coun4.2}.
$(b)$ Graph of the ratio between the PDFs of $U_3(\boldsymbol{r};\boldsymbol{\alpha},\boldsymbol{\sigma},\boldsymbol{\lambda})$ and $U_3(\boldsymbol{r};\boldsymbol{\beta},\boldsymbol{\mu},\boldsymbol{\theta})$ in Counterexample \ref{Coun4.22}.}
\end{center}
\end{figure}

As a follow-up, an illustration is given to validate the result in Theorem \ref{theorem3.2}. 
\begin{example}\label{example4.2}
Consider the left truncated exponential model as the baseline distribution with CDF $F(x)=\frac{e^{-1}-e^{-\frac{x}{2}}}{e^{-1}}$, $x\geq2$. Take $\boldsymbol{r}=(0.30,0.50,0.20)$, $\boldsymbol{s}=(0.85,0.05,0.10)$, 
$\boldsymbol{\sigma}=(2,3,1)$, 
$\boldsymbol{\mu}=(5,4,4.5)$,
$\boldsymbol{\lambda}=(3,5,2)$, $\boldsymbol{\theta}=(7,5,6)$,  
$\boldsymbol{\alpha}=(0.1,5.0,1.3)$, and 
$\boldsymbol{\beta}=(5.1,6.0,5.5)$. Obviously, $\max\left\lbrace\alpha_1,\alpha_2,\alpha_3\right\rbrace\leq \min\left\lbrace\beta_1,\beta_2,\beta_3\right\rbrace$, $\max\left\lbrace\sigma_1,\sigma_2,\sigma_3\right\rbrace\leq \min\left\lbrace\mu_1,\mu_2,\mu_3\right\rbrace$, $\max\left\lbrace\lambda_1,\lambda_2,\lambda_3\right\rbrace\leq \min\left\lbrace\theta_1,\theta_2,\theta_3\right\rbrace$, and $x\tilde{h}(x)$ is decreasing in $x\geq2$. Now, the graph of the ratio between the CDFs of  $U_3(\boldsymbol{s};\boldsymbol{\beta},\boldsymbol{\mu},\boldsymbol{\theta})$ and $U_3(\boldsymbol{r};\boldsymbol{\alpha},\boldsymbol{\sigma},\boldsymbol{\lambda})$ is plotted in Figure \ref{figure5}. From this figure, we conclude that the ratio of the CDFs of $U_3(\boldsymbol{s};\boldsymbol{\beta},\boldsymbol{\mu},\boldsymbol{\theta})$ and $U_3(\boldsymbol{r};\boldsymbol{\alpha},\boldsymbol{\sigma},\boldsymbol{\lambda})$ is increasing in $x\geq5$, which confirms the established result in Theorem \ref{theorem3.2}.  
\end{example}
	
The subsequent counterexample demonstrates that the condition 
``$\max\left\{\lambda_1,\lambda_2,\lambda_3\right\}\leq\min\left\{\theta_1,\theta_2, \theta_3\right\}$'' is essential, together with the other stated conditions, to establish the reversed hazard rate ordering between the MRVs as presented in Theorem \ref{theorem3.2}.
	
\begin{counterexample}\label{Coun4.3}
Consider the left truncated exponential distribution as the baseline with CDF $F(x)=\frac{e^{-\frac{2}{5}}-e^{-\frac{x}{5}}}{e^{-\frac{2}{5}}}$, $x\geq2$. Assume that $\boldsymbol{r}=(0.30,0.50,0.20)$, $\boldsymbol{s}=(0.85,0.05,0.10)$, 
$\boldsymbol{\sigma}=(2,3,4)$, 
$\boldsymbol{\mu}=(6,4,4.5)$,
$\boldsymbol{\lambda}=(3,5,2)$, 
$\boldsymbol{\theta}=(1,2,6)$,  
$\boldsymbol{\alpha}=(0.1,5.0,1.3)$, and 
$\boldsymbol{\beta}=(5.1,6.0,5.5)$. It is obvious that $\max\left\lbrace\alpha_1,\alpha_2,\alpha_3\right\rbrace\leq \min\left\lbrace\beta_1,\beta_2,\beta_3\right\rbrace$, $\max\left\lbrace\sigma_1,\sigma_2,\sigma_3\right\rbrace\leq \min\left\lbrace\mu_1,\mu_2,\mu_3\right\rbrace$, $\max\left\lbrace\lambda_1,\lambda_2,\lambda_3\right\rbrace\nleq \min\left\lbrace\theta_1,\theta_2,\theta_3\right\rbrace$ and $x\tilde{h}(x)$ is decreasing in $x\geq2$. Now, the graph of the ratio between the CDFs of  $U_3(\boldsymbol{s};\boldsymbol{\beta},\boldsymbol{\mu},\boldsymbol{\theta})$ and $U_3(\boldsymbol{r};\boldsymbol{\alpha},\boldsymbol{\sigma},\boldsymbol{\lambda})$ is plotted in Figure \ref{figure6}. From the figure, it can be observed that the ratio exhibits non-monotonic behavior for $x\geq8$, thereby contradicting the conclusion stated in Theorem \ref{theorem3.2}.   
\end{counterexample}
	
\begin{figure}[h!]
\begin{center}
\subfigure[]{\label{figure5}\includegraphics[width=3.2in]{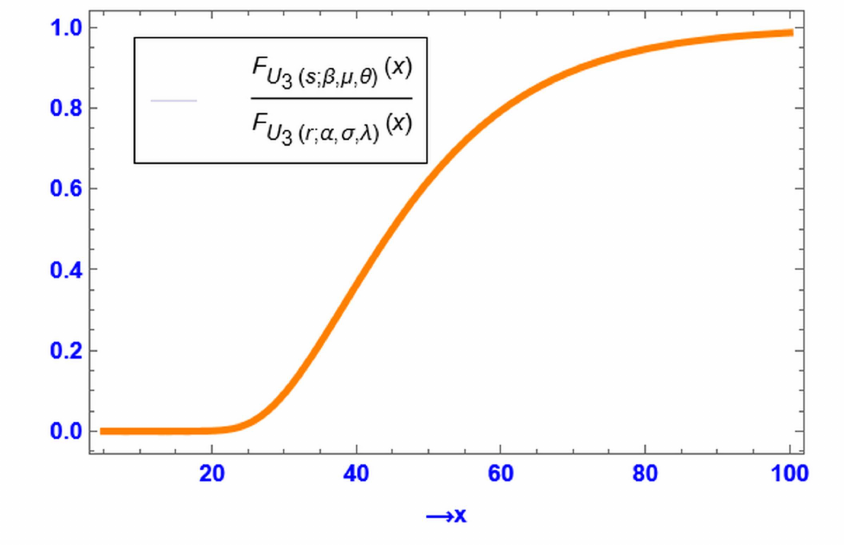}}
\subfigure[]{\label{figure6}\includegraphics[width=3.2in]{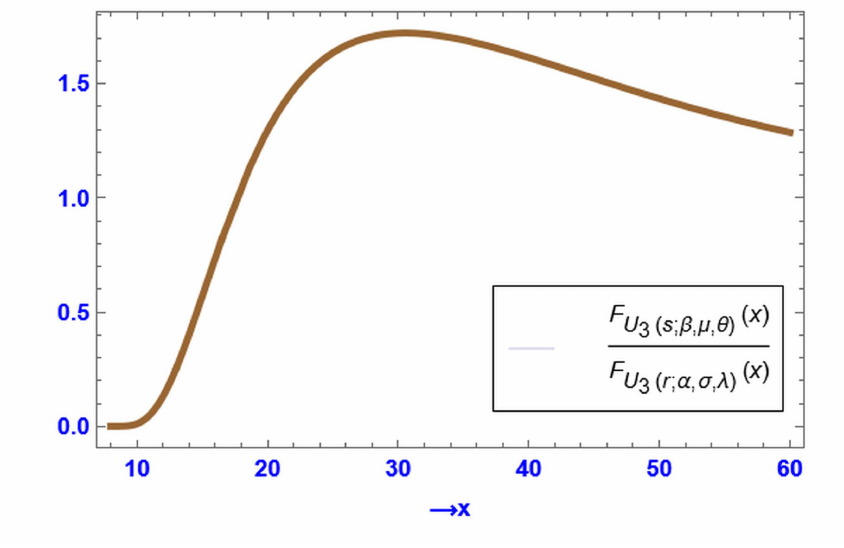}}
\caption{$(a)$ Plot of the ratio between the CDFs of  $U_3(\boldsymbol{s};\boldsymbol{\beta},\boldsymbol{\mu},\boldsymbol{\theta})$ and $U_3(\boldsymbol{r};\boldsymbol{\alpha},\boldsymbol{\sigma},\boldsymbol{\lambda})$ in Example \ref{example4.2}. 
$(b)$ Plot of the ratio between the CDFs of  $U_3(\boldsymbol{s};\boldsymbol{\beta},\boldsymbol{\mu},\boldsymbol{\theta})$ and $U_3(\boldsymbol{r};\boldsymbol{\alpha},\boldsymbol{\sigma},\boldsymbol{\lambda})$ in Counterexample \ref{Coun4.3}.}
\end{center}
\end{figure}
	
The next example provides a verification of the conditions require to prove Theorem \ref{theorem3.3}.
	
\begin{example}\label{example4.3}
Consider the Pareto distribution as the baseline distribution with PDF $f(x)=6(4)^6x^{-7},~x\geq4$. Now, set $\boldsymbol{r}=(0.30,0.20,0.50)$, $\boldsymbol{s}=(0.85,0.05,0.10)$, $\sigma=2$, $\lambda=3$, $\boldsymbol{\alpha}=(6,3,5)$, and 
$\boldsymbol{\beta}=(7,7,6)$. It can be easily seen that $\max\left\lbrace\alpha_1,\alpha_2,\alpha_3\right\rbrace\leq \min\left\lbrace\beta_1,\beta_2,\beta_3\right\rbrace$. We have plotted the graph of the ratio between the PDFs of $U_3(\boldsymbol{s};\boldsymbol{\beta},\sigma,\lambda)$ and $U_3(\boldsymbol{r};\boldsymbol{\alpha},\sigma,\lambda)$ in Figure \ref{figure7} which reveals that the ratio is increasing in $x\geq14$, confirming the result in Theorem \ref{theorem3.3}.  
\end{example}

The counterexample below illustrates that the condition ``$\max\left\{\alpha_1, \alpha_2,\alpha_3\right\}\leq\min\left\{\beta_1,\beta_2,\beta_3\right\}$'' is essential, together with the other specified assumptions, for ensuring the validity of the likelihood ratio order between the MRVs discussed in Theorem \ref{theorem3.3}.
	
\begin{counterexample}\label{Coun4.4}
Consider the Pareto distribution as the baseline distribution with PDF $f(x)=2(3)^2x^{-3}$, $x\geq3$. Now, let us take $\boldsymbol{r}=(0.30,0.20,0.50)$, $\boldsymbol{s}=(0.85,0.05,0.10)$, $\sigma=2$, $\lambda=3$,   $\boldsymbol{\alpha}=(0.2,9,0.1)$, and $\boldsymbol{\beta}=(7,0.7,6)$. It is readily observed that $\max\left\lbrace\alpha_1,\alpha_2,\alpha_3\right\rbrace\nleq \min\left\lbrace\beta_1,\beta_2,\beta_3\right\rbrace$. Here, the plot in Figure \ref{figure8}, which depicts the ratio of the PDFs of $U_3(\boldsymbol{s};\boldsymbol{\beta},\sigma,\lambda)$ and $U_3(\boldsymbol{r};\boldsymbol{\alpha},\sigma,\lambda)$, shows that the ratio becomes non-monotonic for $x\geq11$, thereby contradicting Theorem \ref{theorem3.3}.   
\end{counterexample}
	
\begin{figure}[h!]
\begin{center}
\subfigure[]{\label{figure7}\includegraphics[width=3.2in]{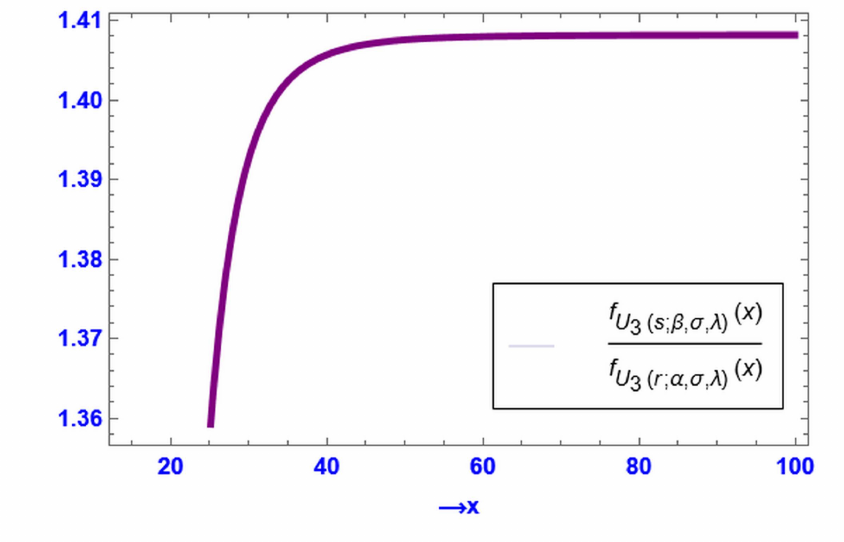}}
\subfigure[]{\label{figure8}\includegraphics[width=3.2in]{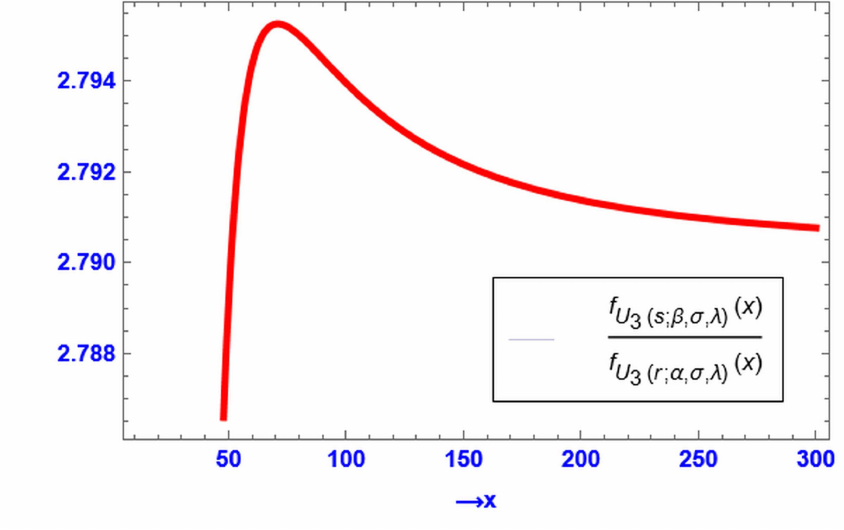}}
\caption{$(a)$ Graph of the ratio between the PDFs of $U_3(\boldsymbol{s};\boldsymbol{\beta},\sigma,\lambda)$ and $U_3(\boldsymbol{r};\boldsymbol{\alpha},\sigma,\lambda)$ in Example \ref{example4.3}. 
$(b)$ Graph of the ratio between the PDFs of $U_3(\boldsymbol{s};\boldsymbol{\beta},\sigma,\lambda)$ and $U_3(\boldsymbol{r};\boldsymbol{\alpha},\sigma,\lambda)$ in Counterexample \ref{Coun4.4}.}
\end{center}
\end{figure}

Next, the following example provides an illustration of the established result in Theorem \ref{theorem3.4}.

\begin{example}\label{example4.4}
Assume the Benktander type-II distribution as the baseline model with CDF $F(x)=1-x^{-0.2}e^{\frac{5}{0.8}(1-x^{0.8})},~x\geq1$. Now, set $\boldsymbol{r}=(0.6,0.3,0.1)$, $\boldsymbol{s}=(0.4,0.4,0.2)$, $\sigma=2$, $\mu=4$, $\lambda=3$, $\theta=5$, $\boldsymbol{\alpha}=(5,6,14)$, and $\boldsymbol{\beta}=(6,9,10)$. It is evident that $\boldsymbol{r},\boldsymbol{s}\in\mathcal{D}_3^+$, $\boldsymbol{\alpha},\boldsymbol{\beta}\in\mathcal{E}_3^+$, ${\boldsymbol{r}}\stackrel{m}{\succcurlyeq}{\boldsymbol{s}}$,  ${\boldsymbol{\alpha}}\stackrel{m}{\succcurlyeq}{\boldsymbol{\beta}}$, $\sigma\leq\mu$,  and $\lambda\leq\theta$. From Figure \ref{figure9}, it is clear that $U_3(\boldsymbol{r};\boldsymbol{\alpha},\sigma,\lambda)\leq_{st}U_3(\boldsymbol{r};\boldsymbol{\beta},\mu,\theta)$, validating the established result in Theorem \ref{theorem3.4}.  
\end{example}
	
In the following counterexample, we illustrate that the condition ``$\boldsymbol{\alpha}\stackrel{m}\succcurlyeq\boldsymbol{\beta}$'' plays a significant role, together with the other conditions in Theorem \ref{theorem3.4}, in establishing the usual stochastic order between two MRVs.
	
\begin{counterexample}\label{Coun5.6}
Consider the Benktander type-II distribution as the baseline distribution with CDF $F(x)=1-x^{-0.7}e^{\frac{2}{0.3}(1-x^{0.3})},~x\geq1$. Now, we choose $\boldsymbol{r}=(0.8,0.1,0.1)$, $\boldsymbol{s}=(0.5,0.3,0.2)$, $\sigma=3$, $\mu=4$, $\lambda=7$, $\theta=7$, $\boldsymbol{\alpha}=(2,4,5)$, and 
$\boldsymbol{\beta}=(1,1.5,3)$. It is easy to see that $\boldsymbol{r},\boldsymbol{s}\in\mathcal{D}_3^+$, $\boldsymbol{\alpha},\boldsymbol{\beta}\in\mathcal{E}_3^+$, $\boldsymbol{r}\stackrel{m}\succcurlyeq\boldsymbol{s}$,  $\boldsymbol{\alpha}\stackrel{m}\not\succcurlyeq\boldsymbol{\beta}$, $\sigma\leq\mu$,  and $\lambda\leq\theta$. From Figure \ref{figure10}, one can observe that 
$U_3(\boldsymbol{r};\boldsymbol{\alpha},\sigma,\lambda)\nleq_{st}U_3(\boldsymbol{r};\boldsymbol{\beta},\mu,\theta)$, which contradicts the assertion of Theorem \ref{theorem3.4}.  
\end{counterexample}

\begin{figure}[h!]
\begin{center}
\subfigure[]{\label{figure9}\includegraphics[width=3.2in]{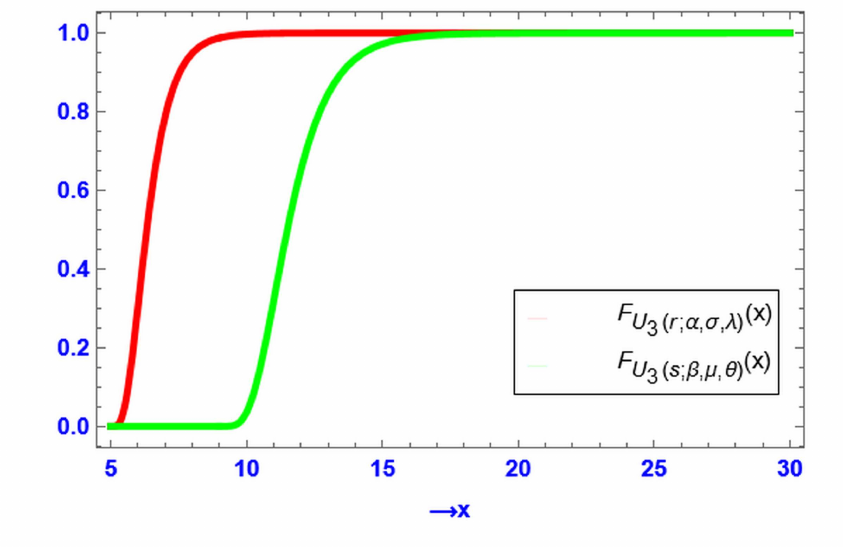}}
\subfigure[]{\label{figure10}\includegraphics[width=3.2in]{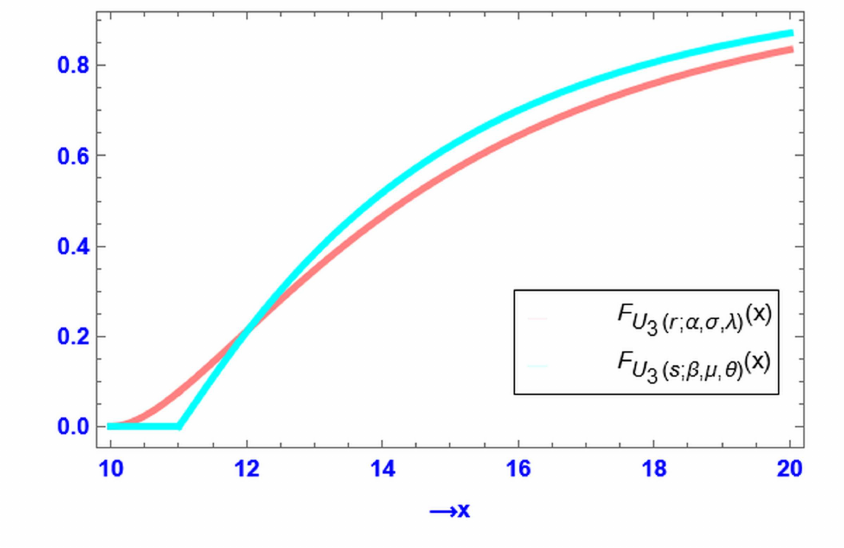}}
\caption{$(a)$ Plots of the CDFs of $U_3(\boldsymbol{r};\boldsymbol{\alpha},\sigma,\lambda)$ (red curve) and $U_3(\boldsymbol{s};\boldsymbol{\beta},\mu,\theta)$ (green curve) in Example \ref{example4.4}. 
$(b)$ Plots of the CDFs of $U_3(\boldsymbol{r};\boldsymbol{\alpha},\sigma,\lambda)$ (pink curve) and $U_3(\boldsymbol{s};\boldsymbol{\beta},\mu,\theta)$ (cyan curve) in Counterexample \ref{Coun5.6}.}
\end{center}
\end{figure}
	
To illustrate Theorem \ref{theorem4.1}, we consider the next example.
\begin{example}\label{example5.5}
Consider the left truncated Burr type-XII distribution as the baseline distribution with CDF $F(x)=\frac{(1+2^{1.5})^{-5}-(1+x^{1.5})^{-5}}{(1+2^{1.5})^{-5}},~x\geq2$. Further, assume that $n_1=25$, $n_2=8$, $n_1^*=15$, $n_2^*=20$, $r_1=0.032$, $r_2=0.025$, $s_1=0.020$, $s_2=0.035$, $\boldsymbol{\sigma}=(5,10)$, $\boldsymbol{\lambda}=(4,6)$, and $\boldsymbol{\alpha}=(2.3,4)$. Here, $n_1r_1n_2^*s_2=0.56\geq0.06=n_2r_2n_1^*s_1$, and $t\tilde{h}(t)$ is decreasing in $t\geq2$. One can clearly observe that $\boldsymbol{\alpha},\boldsymbol{\lambda},\boldsymbol{\sigma}\in\mathcal{E}_2^+$. Figure \ref{figure11} confirms that $U_n(\boldsymbol{r};\boldsymbol{\alpha},\boldsymbol{\sigma},\boldsymbol{\lambda})$ is less than or equal to $U_{n^*}(\boldsymbol{s};\boldsymbol{\alpha},\boldsymbol{\sigma},\boldsymbol{\lambda})$ in the reverse hazard rate order, validating the assertion of Theorem \ref{theorem4.1}. \end{example}
	
The next counterexample shows the importance of the conditions $n_1r_1n_2^*s_2\geq n_2r_2n_1^*s_1$ and $\boldsymbol{\alpha}\in\mathcal{E}_2^+$ to establish the reversed hazard rate ordering between two MRVs in Theorem \ref{theorem4.1}. 
	
\begin{counterexample}\label{Coun5.7}
Let the left truncated Burr type-XII distribution as the baseline distribution with CDF $F(x)=\frac{(1+3^{2})^{-1}-(1+x^{2})^{-1}}{(1+3^{2})^{-1}},~x\geq3$. Suppose we take $n_1=10$, $n_2=10$, $n_1^*=7$, $n_2^*=3$, $r_1=0.03$, $r_2=0.04$, $s_1=0.01$, $s_2=0.01$, $\boldsymbol{\sigma}=(4,9)$, $\boldsymbol{\lambda}=(0.6,8)$, and
$\boldsymbol{\alpha}=(6,2)$. Here, $n_1r_1n_2^*s_2=0.09\ngeq0.28=n_2r_2n_1^*s_1$, and $t\tilde{h}(t)$ is decreasing in $t\geq3$. It follows directly that $\boldsymbol{\lambda},\boldsymbol{\sigma}\in\mathcal{E}_2^+$ but $\boldsymbol{\alpha}\not\in\mathcal{E}_2^+$. As shown in Figure \ref{figure12}, the reverse hazard rate order relation $U_n(\boldsymbol{r};\boldsymbol{\alpha},\sigma,\lambda)\nleq_{rh}U_{n^*}(\boldsymbol{s};\boldsymbol{\alpha},\boldsymbol{\sigma},\boldsymbol{\lambda})$ does not hold, opposing the conclusion of Theorem \ref{theorem4.1}.  
\end{counterexample}

\begin{figure}[h!]
\begin{center}
\subfigure[]{\label{figure11}\includegraphics[width=3.2in]{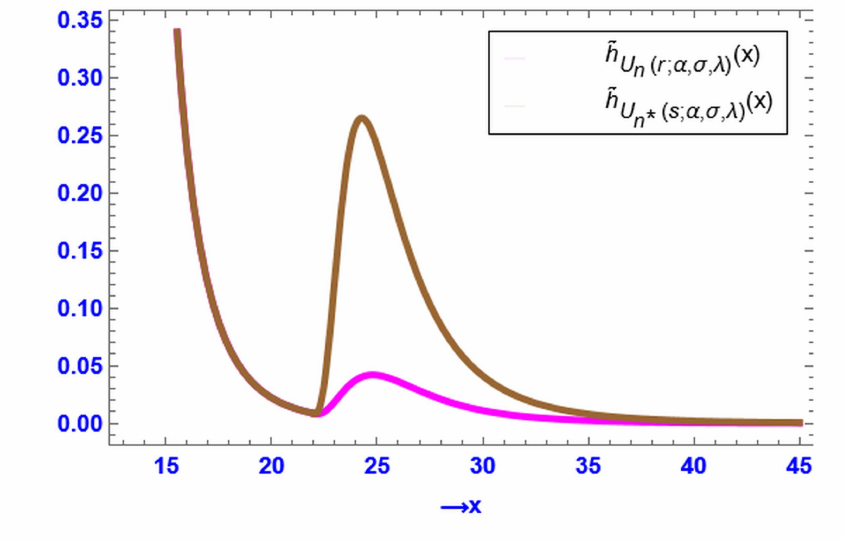}}
\subfigure[]{\label{figure12}\includegraphics[width=3.2in]{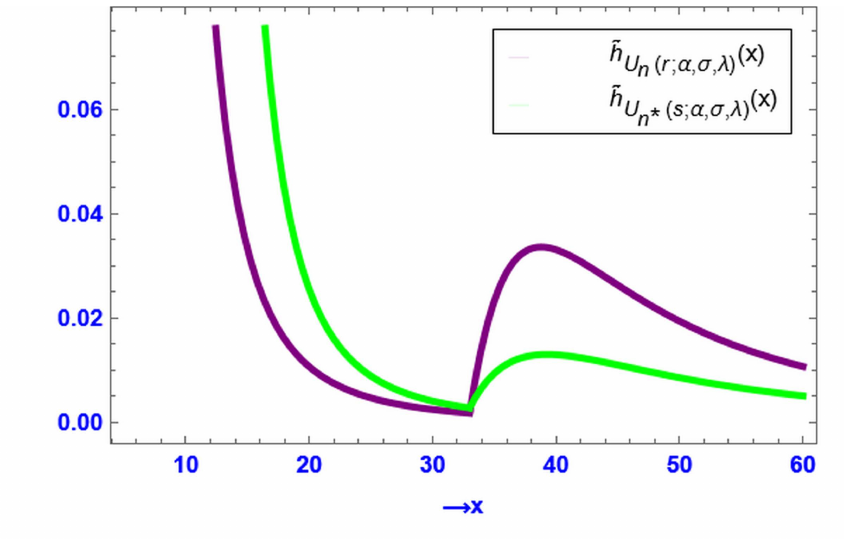}}
\caption{$(a)$ Plots of the reversed hazard rates of $U_n(\boldsymbol{r};\boldsymbol{\alpha},\sigma,\lambda)$ (magenta curve) and $U_{n^*}(\boldsymbol{s};\boldsymbol{\alpha},\sigma,\lambda)$ (brown curve) in Example \ref{example5.5}. 
$(b)$ Plots of the reversed hazard rates of $U_n(\boldsymbol{r};\boldsymbol{\alpha},\boldsymbol{\sigma},\boldsymbol{\lambda})$ (purple curve) and $U_{n^*}(\boldsymbol{s};\boldsymbol{\alpha},\boldsymbol{\sigma},\boldsymbol{\lambda})$ (green curve) in Counterexample \ref{Coun5.7}.}
\end{center}
\end{figure}
	
The next example provides an illustration of Theorem \ref{theorem4.2}.
	
\begin{example}\label{example5.6}
Consider the left truncated Lomax distribution as the baseline distribution with CDF $F(x)=\frac{(1+6)^{-5}-(1+x)^{-5}}{(1+6)^{-5}},~x\geq6$. Set $n_1=15$, $n_2=5$, $n_1^*=10$, $n_2^*=20$, $r_1=0.04$, $r_2=0.08$, $s_1=0.08$, $s_2=0.01$, $\boldsymbol{\sigma}=(3,4)$, $\boldsymbol{\lambda}=(1,2)$, and $\boldsymbol{\alpha}=(2,4)$. Here, $n_1r_1n_2^*s_2=0.12\leq0.32=n_2r_2n_1^*s_1$, $t\tilde{h}(t)$, and $\frac{tf^{\prime}(t)}{f(t)}$ is decreasing in $t\geq6$. From this, it is obvious that $\boldsymbol{\alpha},\boldsymbol{\lambda},\boldsymbol{\sigma}\in\mathcal{E}_3^+$ and $\alpha_1,\alpha_2\geq1$. In this case, we have plotted the graph of the ratio between the PDFs of  $U_n(\boldsymbol{r};\boldsymbol{\alpha},\boldsymbol{\sigma},\boldsymbol{\lambda})$ and $U_{n^*}(\boldsymbol{s};\boldsymbol{\alpha},\boldsymbol{\sigma},\boldsymbol{\lambda})$ in Figure \ref{figure13}, which reveals that the ratio is increasing in $x\geq6$, confirming the result in Theorem \ref{theorem4.2}.  
\end{example}
	
We now present a counterexample to emphasize that the condition ``$\alpha_1,\alpha_2\geq 1$'' is required for the result in Theorem \ref{theorem4.2}.
	
\begin{counterexample}\label{Coun5.8}
Consider the left truncated Lomax distribution as the baseline distribution with CDF $F(x)=\frac{(1+2)^{-3}-(1+x)^{-3}}{(1+2)^{-3}},~x\geq2$. Take $n_1=10$, $n_2=25$, $n_1^*=7$, $n_2^*=3$, $r_1=0.05$, $r_2=0.02$, $s_1=0.01$, $s_2=0.01$, $\boldsymbol{\sigma}=(3,4)$, $\boldsymbol{\lambda}=(2,4)$, and $\boldsymbol{\alpha}=(0.2,0.7)$. Here, $n_1r_1n_2^*s_2=0.15\leq0.35=n_2r_2n_1^*s_1$, $t\tilde{h}(t)$, and $\frac{tf^{\prime}(t)}{f(t)}$ is decreasing in $t\geq2$. It is obvious that $\boldsymbol{\alpha},\boldsymbol{\lambda},\boldsymbol{\sigma}\in\mathcal{E}_3^+$ and $\alpha_1,\alpha_2\ngeq1$. The graph in Figure \ref{figure14} illustrates the ratio of the PDFs between $U_n(\boldsymbol{r};\boldsymbol{\alpha},\boldsymbol{\sigma},\boldsymbol{\lambda})$ and $U_{n^*}(\boldsymbol{s};\boldsymbol{\alpha},\boldsymbol{\sigma},\boldsymbol{\lambda})$, revealing non-monotonicity for $x\geq2$. This observation shows that the likelihood ratio order described in Theorem \ref{theorem4.2} does not hold. 
\end{counterexample}
	
\begin{figure}[h!]
\begin{center}
\subfigure[]{\label{figure13}\includegraphics[width=3.2in]{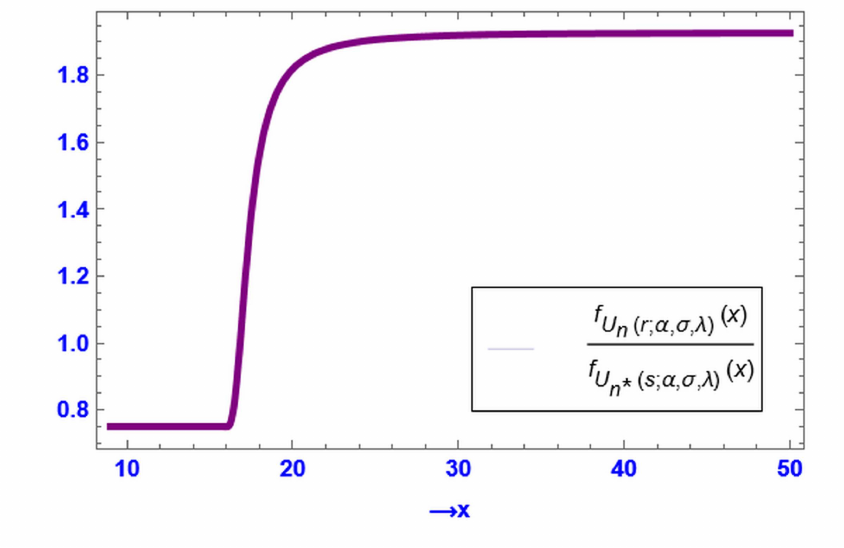}}
\subfigure[]{\label{figure14}\includegraphics[width=3.2in]{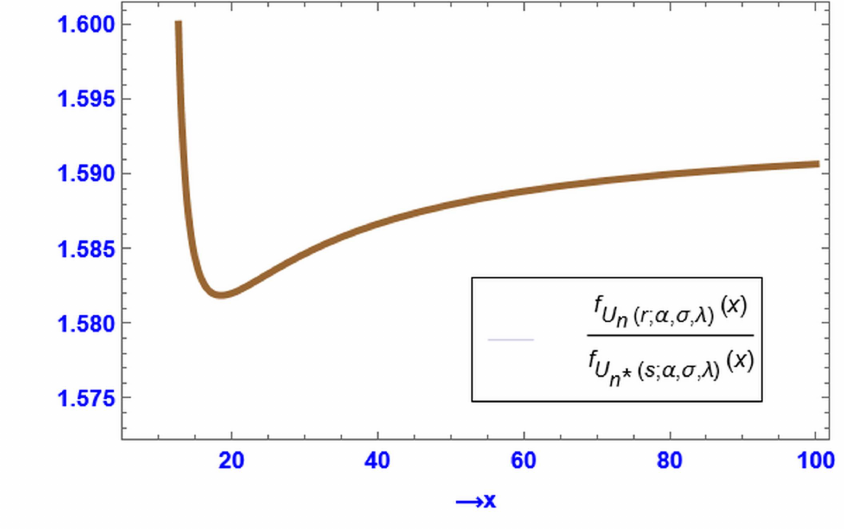}}
\caption{$(a)$ Plot of the ratio of the PDFs of $U_n(\boldsymbol{r};\boldsymbol{\alpha},\boldsymbol{\sigma},\boldsymbol{\lambda})$ and $U_{n^*}(\boldsymbol{s};\boldsymbol{\alpha},\boldsymbol{\sigma},\boldsymbol{\lambda})$  in Example \ref{example5.6}. 
$(b)$ Graph of the ratio of the PDFs of $U_n(\boldsymbol{r};\boldsymbol{\alpha},\boldsymbol{\sigma},\boldsymbol{\lambda})$ and $U_{n^*}(\boldsymbol{s};\boldsymbol{\alpha},\boldsymbol{\sigma},\boldsymbol{\lambda})$  in Counterexample \ref{Coun5.8}.}
\end{center}
\end{figure}
	
The next example provides an illustration of Theorem \ref{theorem4.3}.
	
\begin{example}\label{example5.7}
Consider the log-logistic distribution with CDF $F(x)=\frac{x^{0.9}}{(1+x^{0.9})}$, $x\geq0$ as the baseline model. Take $n_1=10$, $n_2=8$, $n_1^*=20$, $n_2^*=15$, $r_1=0.02$, $r_2=0.10$, $s_1=0.02$, $s_2=0.04$, $\sigma=6$, $\mu=4$, $\boldsymbol{\lambda}=(4,6)$, $\boldsymbol{\theta}=(3,2)$, $\alpha=0.3$. Here, it can be shown that $t\tilde{h}(t)$ is decreasing and $\frac{f^{\prime}(t)}{f(t)}$ is increasing in $t\geq0$. Further, $\sigma\geq\mu$, $\min\{\lambda_1,\lambda_2\}\geq\max\{\theta_1,\theta_2\}$, and $0<\alpha\leq1$. Under this setting, we have plotted the graph of the ratio between reversed hazard rates of  $U_n(\boldsymbol{r};\alpha,\sigma,\boldsymbol{\lambda})$ and $V_{n^*}(\boldsymbol{s};\alpha,\mu,\boldsymbol{\theta})$ in Figure \ref{figure15}, revealing that the ratio is decreasing in $x\geq6$. Thus, Theorem \ref{theorem4.3} is validated.  
\end{example}
	
In the upcoming counterexample, we demonstrate that the condition ``$\frac{f^{\prime}(t)}{f(t)}$ is increasing for $t\geq0$'' plays a key role in obtaining the result in Theorem \ref{theorem4.3}.
	
\begin{counterexample}\label{Coun5.9}
Consider the log-logistic distribution as the baseline distribution with CDF $F(x)=\frac{x^{4}}{(1+x^{4})}$, $x\geq0$. Take $n_1=4$, $n_2=6$, $n_1^*=10$, $n_2^*=25$, $r_1=0.1$, $r_2=0.1$, $s_1=0.05$, $s_2=0.02$, $\sigma=5$, $\mu=2$, $\boldsymbol{\lambda}=(3,7)$, $\boldsymbol{\theta}=(2,1)$, $\alpha=0.8$. Here,  $t\tilde{h}(t)$ is decreasing and $\frac{f^{\prime}(t)}{f(t)}$ is not increasing in $t\geq0$. It may readily be concluded that $\sigma\geq\mu$, $\min\{\lambda_1,\lambda_2\}\geq\max\{\theta_1,\theta_2\}$, and $0<\alpha\leq1$. Figure \ref{figure16} presents the graph of the ratio between the reversed hazard rates of $U_n(\boldsymbol{r};\alpha,\sigma,\boldsymbol{\lambda})$ and $V_{n^*}(\boldsymbol{s};\alpha,\mu,\boldsymbol{\theta})$, which is observed to be non-monotone for $x\geq5$. This observation opposes the result stated in Theorem \ref{theorem4.3}. 
\end{counterexample}
	
\begin{figure}[h!]
\begin{center}
\subfigure[]{\label{figure15}\includegraphics[width=3.2in]{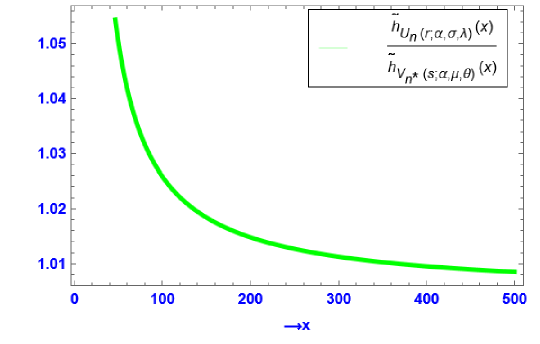}}
\subfigure[]{\label{figure16}\includegraphics[width=3.2in]{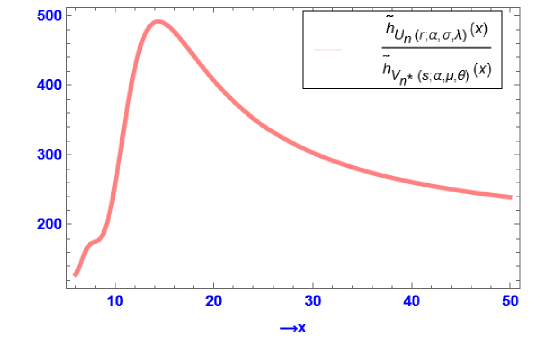}}
\caption{$(a)$ Graphical representation of the ratio between RHRFs of $U_n(\boldsymbol{r};\alpha,\sigma,\boldsymbol{\lambda})$ and $V_{n^*}(\boldsymbol{s};\alpha,\mu,\boldsymbol{\theta})$ in Example \ref{example5.7}. 
$(b)$ Plot of the ratio of the RHRFs of $U_n(\boldsymbol{r};\alpha,\sigma,\boldsymbol{\lambda})$ and $V_{n^*}(\boldsymbol{s};\alpha,\mu,\boldsymbol{\theta})$ in Counterexample \ref{Coun5.9}.}
\end{center}
\end{figure}
	
\section{Concluding remarks}\label{section6}
In this paper, we have considered single-outlier and multiple-outlier FMMs and obtain several ordering results between them. We assumed general family of ELS models for the components of the FMM. Here, Several sufficient conditions have been derived when two single-outlier FMMs are compared in the sense of the usual stochastic order, reversed hazard rate order, and likelihood ratio order. The concept of majorization order has been used in this purpose. Further, the reversed hazard rate order, likelihood ratio order, and AFO in terms of the reversed hazard rate have been established between two multiple-outlier FMMs. To validate the established results, we have provided examples. Few counterexamples have also been presented to show their requirement in establishing the theoretical findings.  
	
Note that our paper is mainly theoretical. However, it is possible to find applications of the established results in reliability practice. To illustrate, consider the following example.
	
Consider two identical components (systems) made by different manufacturers and having several reliability characteristics. Since each of the components are not statistically identical, one can operate in $n$ operational regimes with correspondingly different probabilities $r_i$ and $s_i$, $i\in\mathcal{I}_n$. Assume that the lifetimes of the components in each regime be characterized by distinct distributions with the same functional form as in Theorem \ref{theorem3.4}. The main question is: which of the two components will function better in an appropriate stochastic sense? Upon applying Theorem \ref{theorem3.4}, under the established sufficient conditions, we can thus get the conclusion that the second system performs better than the first system. Similar applications can be found for the other established results.      
	
In addition, the ELS models, when applied within FMMs, offer enhanced flexibility for modeling complex and heterogeneous data. By introducing an additional shape parameter. However, ELS models can better capture skewness, heavy tails, and multi-modal behavior commonly observed in real-world datasets. This results in improved data fitting, more accurate predictions, and greater robustness in applications such as reliability analysis, survival studies, finance, and risk modeling. Moreover, the stochastic properties of ELS mixtures support meaningful comparisons and statistical inference, making them valuable tools for decision-making across various applied fields.      
\\
\\
{\bf Acknowledgements:} Raju Bhakta would like to acknowledge NIIT University, NH-8, Delhi-Jaipur Highway, Neemrana, Rajasthan-301705, India, for providing the necessary research facilities, academic resources, and a supportive environment conducive to scholarly work. The infrastructure and institutional encouragement offered by the university were invaluable in facilitating the successful completion of this research. 
\\
\\
{\bf Conflict of interest statement:} The authors declare that they do not have any conflict of interests.
\normalsize
\bibliography{ref}	

@article{amini2017stochastic,
	title={Stochastic comparisons on two finite mixture models},
	author={Amini-Seresht, Ebrahim and Zhang, Yiying},
	journal={Operations Research Letters},
	volume={45},
	number={5},
	pages={475--480},
	year={2017},
	publisher={Elsevier}
}

@article{balakrishnan2014stochastic,
	title={Stochastic comparisons of series and parallel systems with generalized exponential components},
	author={Balakrishnan, Narayanaswamy and Haidari, Abedin and Masoumifard, Khaled},
	journal={IEEE Transactions on Reliability},
	volume={64},
	number={1},
	pages={333--348},
	year={2014},
	publisher={IEEE}
}

@article{barmalzan2021comparisons,
	title={Comparisons of series and parallel systems with heterogeneous exponentiated geometric components},
	author={Barmalzan, Ghobad and Dehsukhteh, Somayeh Shahraki},
	journal={Communications in Statistics-Theory and Methods},
	volume={50},
	number={18},
	pages={4352--4366},
	year={2021},
	publisher={Taylor \& Francis}
}

@article{barmalzan2022orderings,
	title={Orderings of finite mixture models with location-scale distributed components},
	author={Barmalzan, Ghobad and Kosari, Sajad and Balakrishnan, Narayanaswamy},
	journal={Probability in the Engineering and Informational Sciences},
	volume={36},
	number={2},
	pages={461--481},
	year={2022},
	publisher={Cambridge University Press}
}

@article{bhakta2024stochasticorderings,
	title={Stochastic Orderings between Two Finite Mixtures with Inverted-{K}umaraswamy Distributed Components},
	author={Bhakta, Raju and Kundu, Pradip and Kayal, Suchandan and Alizadeh, Morad},
	journal={Mathematics},
	volume={12},
	number={6},
	pages={852},
	year={2024},
	publisher={MDPI}
}

@article{bhakta2024stochasticcomparisons,
	title={Stochastic comparisons of two finite mixtures of general family of distributions},
	author={Bhakta, Raju and Majumder, Priyanka and Kayal, Suchandan and Balakrishnan, Narayanaswamy},
	journal={Metrika},
	volume={87},
	number={6},
	pages={681--712},
	year={2024},
	publisher={Springer}
}

@article{cha2013failure,
	title={The failure rate dynamics in heterogeneous populations},
	author={Cha, Ji Hwan and Finkelstein, Maxim},
	journal={Reliability Engineering and System Safety},
	volume={112},
	pages={120--128},
	year={2013},
	publisher={Elsevier}
}

@misc{everitt1981finite,
	title={Finite {M}ixture {D}istributions, {C}hapman and {H}all},
	author={Everitt, BS and Hand, DJ},
	year={1981},
	publisher={New York}
}

@book{finkelstein2008failure,
	title={Failure {R}ate {M}odelling for {R}eliability and {R}isk},
	author={Finkelstein, Maxim},
	year={2008},
	publisher={Springer Science \& Business Media}
}

@article{finkelstein2006mixture,
	title={On mixture failure rates ordering},
	author={Finkelstein, Maxim and Esaulova, Veronica},
	journal={Communications in Statistics-Theory and Methods},
	volume={35},
	number={11},
	pages={1943--1955},
	year={2006},
	publisher={Taylor \& Francis}
}

@article{franco2014generalized,
	title={Generalized mixtures of {W}eibull components},
	author={Franco, Manuel and Balakrishnan, Narayanaswamy and Kundu, Debasis and Vivo, Juana Mar{\'\i}a},
	journal={Test},
	volume={23},
	pages={515--535},
	year={2014},
	publisher={Springer}
}

@article{hazra2018stochastic,
	title={On stochastic comparisons of finite mixtures for some semiparametric families of distributions},
	author={Hazra, Nil Kamal and Finkelstein, Maxim},
	journal={Test},
	volume={27},
	number={4},
	pages={988--1006},
	year={2018},
	publisher={Springer}
}

@article{karlis2016finite,
	title={Finite mixtures of censored {P}oisson regression models},
	author={Karlis, Dimitris and Papatla, Purushottam and Roy, Sudipt},
	journal={Statistica Neerlandica},
	volume={70},
	number={2},
	pages={100--122},
	year={2016},
	publisher={Wiley Online Library}
}

@article{kocuk2020incorporating,
	title={Incorporating {B}lack-{L}itterman views in portfolio construction when stock returns are a mixture of normals},
	author={Kocuk, Burak and Cornu{\'e}jols, G{\'e}rard},
	journal={Omega},
	volume={91},
	pages={102008},
	year={2020},
	publisher={Elsevier}
}

@article{li2013stochastic,
	title={Stochastic {O}rders in {R}eliability and {R}isk},
	author={Li, Haijun and Li, Xiaohu},
	journal={Honor of Professor Moshe Shaked. Springer, New York},
	year={2013},
	publisher={Springer}
}

@book{marshall2011inequalities,
	title={Inequalities: {T}heory of {M}ajorization and its {A}pplications},
	author={Marshall, Albert W and Olkin, Ingram and Arnold, Barry C},
    year={2011},
	publisher={{S}econd {E}dition, {S}pringer, {N}ew {Y}ork}
}

@book{muller2002comparison,
	title={Comparison {M}ethods for {S}tochastic {M}odels and {R}isks},
	author={M{\"u}ller, Alfred and Stoyan, Dietrich},
	volume={389},
	year={2002},
	publisher={Wiley New York}
}

@article{navarro2016stochastic,
	title={Stochastic comparisons of generalized mixtures and coherent systems},
	author={Navarro, Jorge},
	journal={Test},
	volume={25},
	number={1},
	pages={150--169},
	year={2016},
	publisher={Springer}
}

@article{navarro2017stochastic,
	title={Stochastic comparisons of distorted distributions, coherent systems and mixtures with ordered components},
	author={Navarro, Jorge and del {\'A}guila, Yolanda},
	journal={Metrika},
	volume={80},
	pages={627--648},
	year={2017},
	publisher={Springer}
}

@article{navarro2013stochastic,
	title={Stochastic ordering properties for systems with dependent identically distributed components},
	author={Navarro, Jorge and del Aguila, Yolanda and Sordo, Miguel A and Su{\'a}rez-Llorens, Alfonso},
	journal={Applied Stochastic Models in Business and Industry},
	volume={29},
	number={3},
	pages={264--278},
	year={2013},
	publisher={Wiley Online Library}
}

@article{navarro2008mean,
	title={Mean residual life functions of finite mixtures, order statistics and coherent systems},
	author={Navarro, Jorge and Hernandez, Pedro J},
	journal={Metrika},
	volume={67},
	number={3},
	pages={277--298},
	year={2008},
	publisher={Springer}
}

@article{panja2021stochastic,
	title={On stochastic comparisons of finite mixture models},
	author={Panja, Arindam and Kundu, Pradip and Pradhan, Biswabrata},
	journal={Stochastic Models},
	volume={38},
	number={2},
	pages={190--213},
	year={2022},
	publisher={Taylor \& Francis}
}

@article{sattari2021stochastic,
	title={Stochastic comparisons of finite mixture models with generalized {L}ehmann distributed components},
	author={Sattari, Mostafa and Barmalzan, Ghobad and Balakrishnan, Narayanaswamy},
	journal={Communications in Statistics-Theory and Methods},
	volume={51},
	number={22},
	pages={7767-7782},
	year={2022},
	publisher={Taylor \& Francis}
}

@Book{shaked2007stochastic,
	title     = {Stochastic {O}rders},
	author    = {Shaked, Moshe and Shanthikumar, J. George},
	year      = {2007},
	publisher = {{S}econd {E}dition, {S}pringer, {N}ew {Y}ork}
}
\end{document}